\documentclass[a4paper, reqno]{amsart}

\title[{}]{Center of the affine $\gl_{n|1}$ at the critical level and pseudo-differential operators}

\author{Dražen Adamović}
\address[D.A.]{Faculty of Science, Department of Mathematics, University of Zagreb, Bijenicka 30, 10 000 Zagreb, Croatia}
\email{adamovic@math.hr}

\author{Boris Feigin}
\address[B.F]{Hebrew University of Jerusalem, Einstein Institute of Mathematics, Givat Ram. Jerusalem, 9190401, Israel}
\email{borfeigin@gmail.com}

\author{Shigenori Nakatsuka}
\address[S.N.]{Department Mathematik, FAU Erlangen–Nürnberg, Cauerstraße 11, 91058, Erlangen, Germany}
\email{shigenori.nakatsuka@fau.de}

\usepackage{array}
\usepackage{latexsym}
\usepackage{amsmath}
\usepackage{amsfonts}
\usepackage{amssymb,nccmath}
\usepackage{mathtools}
\usepackage{amsxtra}
\usepackage{mathrsfs}
\usepackage{eucal}
\usepackage[all]{xy}
\usepackage{color}
\usepackage{braket}
\usepackage{graphics, setspace}
\usepackage{etoolbox, textcomp}
\usepackage{comment}
\usepackage{enumitem}
\usepackage{changepage}
\usepackage{xcolor}
\definecolor{rouge}{rgb}{0.85,0.1,.4}
\definecolor{bleu}{rgb}{0.1,0.2,0.9}
\definecolor{violet}{rgb}{0.7,0,0.8}
\usepackage[colorlinks=true,linkcolor=bleu,urlcolor=violet,citecolor=rouge]{hyperref}
\usepackage{cleveref}
\usepackage{lscape}
\usepackage{caption}
\usepackage{subcaption}
\usepackage{enumitem}
\usepackage{graphicx}
\usepackage{arydshln}

\usepackage{tikz}			%dynkin-diagrams
\usetikzlibrary{matrix}
\usetikzlibrary{automata, arrows.meta, positioning,cd}
\usepackage{dynkin-diagrams}
\usetikzlibrary{cd}			%commutative-diagrams

\tikzset{%
  symbol/.style={
    draw=none,
    every to/.append style={
      edge node={node [sloped, allow upside down, auto=false]{$#1$}}
    },
  },
}

\newtheorem{definition}{Definition}[section]
\newtheorem{proposition}[definition]{Proposition}
\newtheorem{theorem}[definition]{Theorem}
\newtheorem{ThmLetter}{Theorem}
\newtheorem{corollary}[definition]{Corollary}
\newtheorem{lemma}[definition]{Lemma}
\newtheorem{remark}[definition]{Remark}
\newtheorem{conjecture}[definition]{Conjecture}

\numberwithin{equation}{section}

\newcommand{\Z}{\mathbb{Z}}
\newcommand{\R}{\mathbb{R}}
\newcommand{\C}{\mathbb{C}}

\newcommand{\rH}{\mathrm{H}}

\newcommand{\End}{\operatorname{End}}

\newcommand{\Com}{\operatorname{Com}}

\newcommand{\Ker}{\operatorname{Ker}}

\renewcommand{\Im}{\operatorname{Im}}

\newcommand{\ch}{\operatorname{ch}}

\newcommand{\str}{\operatorname{str}}

\newcommand{\pd}{\partial}
\newcommand{\gr}{\operatorname{gr}}
\newcommand{\ad}{\operatorname{ad}}
\newcommand{\wt}{\operatorname{wt}}

\newcommand{\A}{\mathbb{A}}

\newcommand{\wun}{\mathbf{1}}
\newcommand{\kk}{\kappa}
\newcommand{\aff}{\mathrm{aff}}
\renewcommand{\geq}{\geqslant}
\renewcommand{\leq}{\leqslant}
\newcommand{\hwt}[1]{\mathrm{e}^{{#1}}}

\newcommand{\HH}{\mathrm{H}}
\newcommand{\DS}{\mathrm{H}_{\mathrm{DS}}}
\newcommand{\D}{\mathcal{D}}
\newcommand{\sA}{\mathcal{A}}
\newcommand{\OO}{\mathbb{O}}
\newcommand{\heis}[1]{\pi_\mathfrak{h}^{#1}}
\newcommand{\w}[1]{\widehat{#1}}
\newcommand{\op}[1]{\check{#1}}
\newcommand{\affWak}[1]{\mathbb{W}^{\kk}_{{#1}}}

\newcommand{\subW}[1]{\W^{#1}(\sll_n,\OO_{\mathrm{sr}})}
\newcommand{\Q}{\mathbb{Q}}

\newcommand{\GL}{\mathrm{GL}}
\newcommand{\g}{\mathfrak{g}}
\newcommand{\h}{\mathfrak{h}}

\newcommand{\gl}{\mathfrak{gl}}
\newcommand{\sll}{\mathfrak{sl}}

\newcommand{\lm}[2]{[#1 {}_\lambda #2]}
\newcommand{\plm}[2]{\{#1 {}_\lambda #2\}}
\newcommand{\cent}{\mathfrak{z}}
\newcommand{\nil}{\mathfrak{n}}
\newcommand{\n}{\mathfrak{n}}
\newcommand{\tori}{\mathfrak{t}}

\newcommand{\tp}{{\scalebox{0.5}{$+$}}}
\newcommand{\tn}{{\scalebox{0.5}{$-$}}}
\newcommand{\tpn}{{\scalebox{0.5}{$\pm$}}}

\newcommand{\dz}{\mathrm{d}z}
\newcommand{\dC}{\mathrm{h}^{\scalebox{0.5}{$\vee$}}}
\newcommand{\res}[1]{\underset{z=0}{\mathrm{Res}} #1\frac{\dz}{z}}
\newcommand{\tmath}[1]{\texorpdfstring{#1}{ABC}}
\newcommand{\dx}{\check{\mathbf{x}}}
\newcommand{\dso}{ {\check{\mathbf{O}}} }
\newcommand{\x}{\mathbf{x}}
\newcommand{\dO}{ {{\mathbf{O}}} }

\newcommand{\sub}{\mathrm{sr}}

\newcommand{\bea}{\begin{eqnarray}}
\newcommand{\eea}{\end{eqnarray}}

\newcommand{\Mod}{\text{-}\mathrm{mod}}

\newcommand{\W}{\mathcal{W}}

\newcommand\doi[2]{\href{http://dx.doi.org/#1}{#2}}

\begin{document}
\maketitle

\begin{abstract}
%{\color{blue} We describe the center of the affine vertex algebra associated to  $\gl_{n|1}$  at the critical level and generalize our recent result on the center in the case $\gl_{2|1}$ from \cite{AN-2024}.}
We prove that the center of the affine Lie algebra $\w{\gl}_{n|1}$ at the critical level is generated by the coefficients in the expansion of the pseudo-differential operator $(\partial_z-u_1(z))\cdots (\partial_z-u_n(z))(\partial_z+u_{n+1}(z))^{-1}$ taking values in the Cartan subalgebra. This is an affine analogue of the Harish-Chandra isomorphism in the finite case.

The key ingredient of the proof is the identification of the center with the Heisenberg coset of the regular $\W$-superalgebra of $\gl_{n|1}$ at the critical level, whose associated graded algebra is realized as the affine supersymmetric polynomials.
Based on this, we derive a character formula for the center, which coincides with the generating function of plane partitions with a pit
condition.
We also prove that the Heisenberg coset at generic levels has a similar interpretation in terms of pseudo-differential operators that deform the one at the critical level.

\end{abstract}

%\tableofcontents

\section{Introduction}\label{Intro}

The description of the center of the universal enveloping algebra
of a finite-dimensional reductive Lie algebra has played a crucial role in the study of representation theory. The Harish-Chandra isomorphism identifies the center
$\cent(\g)$ with the algebra of Weyl group invariants in the symmetric
algebra of a Cartan subalgebra $\h$. 
This provides fundamental insights into the structure of the BGG category,
such as the linkage principle.

The affine analogue was conjectured by Drinfeld and established by one of the authors and E. Frenkel \cite{FF1}. 
More precisely, one considers the affine Lie algebras $\w{\g}_\kk$ at level $\kk$ and studies the center of the completed enveloping algebra $U_\kk(\w{\g})$. The completion reflects the fact that $\w{\g}_\kk$ is a one-dimensional central extension of the loop algebra $\g(\!(z)\!)$ parametrized by the level $\kk$. The center turns out to be non-trivial at one distinguished level, called the critical level $\kk_c$. When $\kk=\kk_c$, the center is identified with the regular functions on the moduli of certain differential operators, called $\check{\g}$-opers, over the punctured disc, where $\check{\g}$ is the Langlands dual of $\g$.
This gives a ``spectral decomposition" of the category of smooth $\w{\g}_{\kk_c}$-representations parametrized by $\check{\g}$-opers, which was one of the earliest results on the geometric Langlands program \cite{F}.

The category of smooth $\w{\g}_{\kk_c}$-modules is canonically identified with the category of modules over the affine vertex algebra $V^{\kk_c}(\g)$. 
Indeed, the formulation of vertex algebras fits well in this context: both $U_{\kk_c}(\w{\g})$ and its center are naturally identified with the enveloping algebras of $V^{\kk_c}(\g)$ and its center $\cent(V^{\kk_c}(\g))$ as vertex algebras. In this formulation, we have the affine analogue of the Harish-Chandra isomorphism
\begin{align*}
    \cent(V^{\kk_c}(\g)) \simeq \W^{\kk_c}(\g)\subset \pi^0_\h,
\end{align*}
where $\pi^0_\h$ is the Heisenberg vertex algebra associated with the Cartan subalgebra $\h$ at level zero, and $\W^{\kk_c}(\g)$ is the subalgebra characterized by the joint kernel of screening operators, the analogue of the Weyl group invariants in the finite setting. 
When $\g=\gl_n$, $\W^{\kk_c}(\gl_n)$ is generated by the coefficients of the expansion of the differential operator 
$$(\partial_z- u_1(z))\cdots (\partial_z-u_{n-1}(z))(\partial_z- u_n(z))$$
or symbolically $(\partial- u_1)\cdots (\partial-u_{n-1})(\partial- u_n)$ in the language of vertex algebras where $u_i$'s are the basis of the Cartan subalgebra $\h$.  
The algebra $\W^{\kk_c}(\g)$ is called \emph{the $\W$-algebra} defined as the BRST reduction of $V^\kk(\g)$
$$\W^{\kk}(\g)=\HH_{\OO_{\mathrm{reg}}}^0(V^\kk(\g))$$
associated with the regular nilpotent orbit $\OO_{\mathrm{reg}}\subset \g$ at the critical level $\kk=\kk_c$.
This is the quantum affine analogue of the Kostant slice \cite{Kos78}, which is naturally isomorphic to the GIT quotient $\g/\!/G$ of the action of the adjoint group $G$ on $\g$. 

The $\W$-algebras enjoy the duality, known as the Feigin-Frenkel duality \cite{FF1}, 
$$\W^{\kk}(\g)\simeq \W^{\check{\kk}}(\check{\g})$$
where the levels $\kk,\check{\kk}$ satisfy a duality relation. 
This is a deformation of the one ar the critical level $(\kk,\check{\kk})=(\kk_c,\infty)$, which recovers the description of $\cent(V^{\kk_c}(\g))$ by $\check{\g}$-opers mentioned above, and plays an important role in the quantum geometric Langlands program e.g. \cite{AF, Gaits}.

While the representation theory of the $\W$-algebras $\W^{\kk}(\g)$ serves as the Whittaker models of $\w{\g}_\kk$-modules, the generalization to $\W$-algebras $\W^{\kk}(\g,\OO)$ associated with general nilpotent orbits $\OO$ are related to degenerate Whittaker models. 
The generalization of the Feigin-Frenkel duality for $\W^{\kk}(\g,\OO)$ or its geometric/categorical analogues has been studied in various contexts of mathematics and physics (e.g. \cite{BFGT, CGN21, CL4, GR}). It turns out that they meet the $\W$-superalgebras $\W^{\kk}(\g,\OO)$ associated with basic-classical Lie \emph{superalgebras}. This motivates revisiting the problem of describing the center $\cent(V^{\kk_c}(\g))$ via differential operators, which has remained open for decades for basic-classical Lie superalgebras. 

\vspace{1ex}
The main results of this paper, below, resolve this problem by using \emph{pseudo}-differential operators and establish related combinatorial results for the affine Lie superalgebras $\w{\gl}_{n|1}$, which are the first nontrivial family in the super case.

\begin{ThmLetter}[Theorem \ref{them: Detecting the FF center}/\ref{THM: FF center for gln1 vis pD op}/\ref{thm: character formula of the center}]\label{thm: A} \hspace{0mm}\\
\begin{enumerate}[leftmargin=*, labelsep=1em]
    \item There is an embedding of vertex algebras 
$$\cent(V^{\kk_c}(\gl_{n|1}))\hookrightarrow \pi_\h^0,$$
whose image is generated by the coefficients of the pseudo-differential operator
$$(\partial_z- u_1(z))\cdots (\partial_z- u_n(z)) (\partial_z+u_{n+1}(z))^{-1}.$$
 \item There is a filtration on $\cent(V^{\kk_c}(\gl_{n|1}))$ whose associated graded differential algebra is isomorphic to the algebra of affine supersymmetric polynomials:
 \begin{align*}
 \gr\cent(V^{\kk_c}(\gl_{n|1}))\simeq \Lambda^{n|1}_{\aff}.
\end{align*}
 \item The $q$-character of the center $\cent(V^{\kk_c}(\gl_{n|1}))$ is given by
    \begin{align*}
    \ch[\cent(V^{\kk_c}(\gl_{n|1}))](q)=\frac{\sum_{m=0}^\infty (-1)^mq^{\frac{1}{2}m(m+1)}(q^{m+1};q)_{n-1}}{(q;q)_\infty^2(q,\cdots,q^{n-1};q)_\infty}.
\end{align*}
\end{enumerate}
\end{ThmLetter}

We note that (2) was conjectured by Molev and Mukhin \cite{MM} (established for $n=1,2$ \cite{AN-2024, MM}) and implies that the center $\cent(V^{\kk_c}(\gl_{n|1}))$ is strongly generated by higher Segal-Sugawara vectors constructed by Molev and Ragoucy \cite{MR}.
The $q$-series appearing on the right-hand side in (3) agrees with the generating function of \emph{plane partitions} with the pit condition $(2,n+1)$ (Corollary \ref{cor: characters as plane partition counting}). This confirms a conjecture by Molev and Mukhin, generalizing the case $n=1$ \cite{A, MM}.

The key observation of the proof of Theorem \ref{thm: A} is the identification of the centers 
$$\cent(V^{\kk_c}(\gl_{n|1})) \simeq \cent(\W^{\kk_c}(\gl_{n|1}))$$
(Theorem \ref{them: Detecting the FF center}). This is a generalization of the case $n=2$ in our earlier work \cite{AN-2024}, which played the main role in identifying $\cent(V^{\kk_c}(\gl_{2|1}))$ with the large level limit of the parafermion vertex algebra $K^\infty(\gl_2)$ of $\gl_2$. We note that the identification of the center $\cent(\W^{\kk_c}(\g,\OO))$ with $\cent(V^{\kk_c}(\g))$ in the non-super case is well-known in general \cite{Ar12b, AM21}. 
The center $\cent(\W^{\kk_c}(\gl_{n|1}))$ can be \emph{deformed} to other levels, as it is the specialization of the Heisenberg coset subalgebra
\begin{align*}
    C^\kk(\gl_{n|1})=\Com\left(\pi_J,\W^\kk(\gl_{n|1}) \right)\subset \W^\kk(\gl_{n|1})
\end{align*}
at the critical level (Theorem \ref{thm: Wakimoto for sCoset}), see Section \ref{sec: Heisenberg coset} for details. This coset subalgebra $C^\kk(\gl_{n|1})$ at generic levels has several properties similar to those of the regular $\W$-algebras, such as a variant of the Feigin-Frenkel duality \cite{CGN21, FS}.
Indeed, the results in Theorem \ref{thm: A} can be deformed outside the critical levels. 
Among them is the following result originally conjectured by T. Procházka in physics (e.g. \cite{PR}).
\begin{ThmLetter}[Theorem \ref{prop: coset for Wsalg}]\label{thm: main B}
There is an embedding of vertex algebras 
$$C^\kk(\gl_{n|1})\hookrightarrow \pi_{\tori}^{k+\dC}$$
for $k \neq -\dC+1$ whose image at the generic levels is generated by the coefficients of the pseudo-differential operator 
$$(\partial_z+\varepsilon t_1(z))\cdots (\partial_z+\varepsilon t_n(z))(\partial_z-\varepsilon t_{n+1}(z))^{\varepsilon},$$
where $\varepsilon=1/(-1+k+\dC)$.
\end{ThmLetter}
This result implies that the center $\cent(V^{\kk_c}(\gl_{n|1}))$ realized as the regular functions of the pseudo-differential operator in Theorem \ref{thm: A} (1) inherits a natural Poisson vertex algebra structure restricting on the standard one on $\pi^0_{\h}$ (Proposition \ref{prop: Poisson str on the center}).

\vspace{1ex}
In future works, we plan to describe the center $\cent(V^{\kk_c}(\g))$ for basic-classical (simple) Lie superalgebras through pseudo-differential operators (hopefully related to superopers \cite{Z} and Gaudin models \cite{LM21, MVY}) and connect the category of spherical $V^{\kk_c}(\g)$-representations to quasi-coherent sheaves over such operators \cite{FG}. 

\vspace{1ex}
\paragraph{\textbf{Outline.}} 
The rest of the paper is organized as follows.
We start in Section \ref{sec: Wsuperalgebras} with a review of affine vertex superalgebras and $\W$-superalgebras.
In Section \ref{Wsalg and Heisenberg coset}, we study the Heisenberg coset of the regular $\W$-superalgebras $\W^\kk(\gl_{n|1})$ and prove Theorem \ref{thm: main B}. 
In Section \ref{sec: duality and strong generators}, we determine the strong generators of $\W^\kk(\gl_{n|1})$ based on the Feigin-Semikhatov duality, and describe the center of $\W^\kk(\gl_{n|1})$ at the critical level $\kk=\kk_c$. In Section \ref{sec: Wakimoto realization}, we review the Wakimoto realization of $\W$-superalgebras and relate the centers of $\W^{\kk_c}(\gl_{n|1})$ and $V^{\kk_c}(\gl_{n|1})$.
Based on them, we prove Theorem \ref{thm: A} in Section \ref{sec: Center of affine at the critical level} and \ref{sec: Character formula}.
We formulate a conjecture on the structure of regular $\W$-superalgebras $\W^\kk(\gl_{n|m})$ at the critical level and show the first non-trivial example in Section \ref{sec: Some conjectures}.

\vspace{1ex}
\paragraph{\textbf{Acknowledgements.}} 
We would like to thank T. Arakawa and M. Gorelik for their interest throughout this work.
S.N. thanks T. Procházka for useful discussion. 
D.A. is partially supported by the Croatian Science Foundation under the project IP-2022-10-9006 and   by the project “Implementation of cutting-edge research and its application as part of the Scientific Center of
Excellence for Quantum and Complex Systems, and Representations of Lie Algebras“, Grant No. PK.1.1.10.0004, co-financed by the European Union through the European Regional Development Fund--Competitiveness and Cohesion Programme.
2021-2027.
B.F. is partially supported by ISF 3037 2025 1 1.

\section{\tmath{$\W$}-superalgebras}\label{sec: Wsuperalgebras}
\subsection{Affine vertex superalgebras} Let $\g=\gl_{m|n}$ be the general linear Lie superalgebra associated with the vector superspace $V=\C^{m|n}$ and $e_{i,j}$ ($i,j=1,\dots, m+n$) be the $(i,j)$-th elementary matrix. 
We take the even supersymmetric invariant bilinear form on $\gl_{m|n}$ to be 
\begin{align*}
    \kk=(k+\dC)\kk_{V}- \frac{1}{2} \kk_{\g}
\end{align*}
where $\dC=n-m$ is the dual Coxeter number, and $\kk_V$ and $\kk_\g$ are the supertrace on $V$ and the Killing form, that is,
\begin{align*}
    \kk_V(u,v)=\str_V(uv),\quad \kk_\g=\str_\g(\ad_u \ad_v ).
\end{align*}
Let us identify $\kk$ with $k$ in the below. 
The affine Lie algebra $\widehat{\g}$ is the central extension of the loop algebra $\g[t^{\pm1}]$
\begin{align*}
    0\rightarrow \C K \rightarrow \widehat{\g}\rightarrow \g[t^{\pm1}] \rightarrow 0
\end{align*}
with Lie bracket 
\begin{align*}
    [u_{(p)},v_{(q)}]:=[u,v]_{(p+q)}+\kk(u,v) p\delta_{p+q,0}K,\quad [K,\widehat{\g}]=0
\end{align*}
where $u_{(p)}=ut^p$.
The (universal) affine vertex superalgebra associated with $\g$ at level $\kk$ is defined as the parabolic induction 
\begin{align*}
    V^\kk(\g):=U(\widehat{\g})\otimes_{U(\g[t]\oplus \C K)} \C
\end{align*}
induced from the $(\g[t]\oplus \C K)$-module $\C$ where $\g[t]$ acts by $0$ and $K$ by the identity, equipped with a unique vertex algebra structure with the unit $\wun=1\otimes1$, the translation operator $\pd$ satisfying $[\pd,u_{(p)}]=-p u_{(p-1)}$, and the state-field correspondence $Y(\cdot,z)$ satisfying $Y(u,z)=\sum u_{(p)} z^{-p-1}$ were $u$ is identified with $u_{(-1)}\otimes \wun$. It will be convenient to use the $\lambda$-bracket 
$$[a{}_\lambda b]=\sum_{n=0}^{\infty} \frac{1}{n!}a_{(n)}b\lambda^n$$
instead of the commutator formula 
$$[Y(a,z),Y(b,w)]=\sum_{n=0}^\infty Y(a_{(n)}b,w)\frac{1}{n!}\partial_w^n\delta(z-w)$$
by using the delta function $\delta(z-w)=\sum_{n=-\infty}^\infty z^nw^{-n-1}$.

Note that for $m \neq n$ as we will consider otherwise stated, the decomposition $\gl_{m|n}=\sll_{m|n}\oplus \C \Omega$ with $\Omega=\sum e_{i,i}$ induces an isomorphism of vertex superalgebras
\begin{align}\label{eq: decomp of gl}
    V^\kk(\gl_{m|n})\simeq  V^\kk(\sll_{m|n}) \otimes  \pi^{\kk}_{\Omega}.
\end{align}
Here $\pi^{\kk}_{\Omega}$ is the Heisenberg vertex algebra $V^{\kk}(\gl_1)$ with $\gl_1=\C \Omega$ as we have $\lm{u}{\Omega}=0$ for all $u\in \sll_{m|n}$. On the other hand, 
\begin{align*}
    &\lm{u}{v}=[u,v]+k\kk_{V}(u,v),\quad \lm{\Omega}{\Omega}=(m-n)(k+\dC)\lambda
\end{align*}
hold for $u,v \in \sll_{m|n}$.

The level $\kk=-\frac{1}{2}\kk_\g$, denoted by $\kk_c$ is called \emph{the critical level}; the bilinear form is explicitly given by  
\begin{align*}
    \kappa_c(e_{i,j},e_{p,q})=\delta_{i,j}\delta_{p,q}(-1)^{p(i)+p(q)}-(m-n)\delta_{i,j}\delta_{p,q}(-1)^{p(i)}
\end{align*}
where $p(i)$ is the parity, which is 0 (resp. 1) if $i\leq m$ (resp. $i>m$).

Given a vertex superalgebra $V$ and a subalgebra $W$, the coset $\Com(W,V)$ is the subalgebra defined by 
\begin{align*}
    \Com(W,V):=\{u \in V ; [Y(u,z),Y(v,w)]=0\ (\forall v\in W)\}.
\end{align*}
In particular, the coset subalgebra $\cent(V):=\Com(V,V)$
is called the center of $V$. 
For a simple Lie algebra $\g$, the center $\mathfrak{z}(V^\kk(\g))$ is known to be trivial $\C$ at $\kk\neq \kk_c$, whereas non-trivial at the critical level $\kk=\kk_c$.
The difference comes from the existence of the so-called Segal-Sugawara conformal vector $L$ given by 
$$L=\frac{1}{k+\dC}S,\qquad S=\frac{1}{2}\sum_{i=1}^{\dim \g} (-1)^{p(x_i)}x_ix^i$$ 
where $x_i$ is a basis of $\g$ with dual basis $x^i$ satisfying $\kk_V(x_i,x^j)=\delta_{i,j}$. 
While $S$ is defined for all levels and satisfies 
$$[S {}_\lambda u]=(k+\dC)(\partial u+ u \lambda)$$
for all $u \in \g$, $L$ is defined only outside the critical level.
The above formula implies that $S$ is central at $\kk=\kk_c$ while the rescaled one $L$ provides a conical internal grading structure, which implies $\cent(V^\kappa(\g))=\C$ for all $\kk \neq \kk_c$.

\subsection{\tmath{$\W$}-superalgebras}
In the finite case, the description of the center $\cent(\g)\subset U(\g)$ for basic-classical simple Lie superalgebras $\g$ is obtained through the Harish-Chandra homomorphism 
$$\xi\colon U(\g) \twoheadrightarrow U(\h)$$
where $\h$ is a Cartan subalgebra, and is called the Harish-Chandra isomorphism, see \cite{Musson}. 
In the even case, it gives rise to an isomorphism $\cent(\g)\simeq U(\h)^W$ where $W$ is the Weyl group. An alternative approach is to quantize the Kostant slice, which is defined as the subvariety $\mathscr{S}=f+\g^e\subset \g$ by using the regular (a.k.a principal) $\sll_2$-triple $\{e,h,f\}\subset \g$.
Such a triple exists uniquely up to conjugation by the adjoint group $G$ whose Lie algebra is $\g$. Then, $\mathscr{S}$ is isomorphic to the categorical quotient $\g/\!/G$. The slice $\mathscr{S}$ is obtained as a Hamiltonian reduction of $\g$ for the unipotent group $N_+$ associated with the upper nilpotent subalgebra $\nil_+$ and the character $\chi=\kappa_V(f,-)\colon \nil_+\rightarrow \C$. 
The Hamiltonian reduction is quantized to the so-called BRST reduction of $U(\g)$ to obtain $\cent(\g)$, whose vertex algebraic analogue was introduced \cite{FF1} to describe the center $\cent(V^{\kk_c}(\g))$, which by construction can be deformed to other levels and called the regular $\W$-algebras $\W^\kk(\g)$.

More generally, the $\W$-superalgebras are vertex superalgebras \cite{KRW03, KW04} obtained from affine vertex superalgeras $V^\kk(\g)$ associated with basic-classical Lie superalgebras $\g$ at level $\kk$ through the BRST reduction associated with an $\sll_2$-triple $\{e,h,f\}$ and a good $\frac{1}{2}\Z$-grading $\Gamma\colon \g=\oplus_d \g_d$. For simplicity, we consider the case where the grading $\Gamma$ is called even, i.e., a $\Z$-grading. 
Let us consider $\g_+=\bigoplus_{d>0}\g_d$ and take a (parity-homogeneous) basis $e_\alpha$ ($\alpha \in S_+$) with structure constants $c_{\alpha\beta}^\gamma$, i.e.  $[e_\alpha,e_\beta]=\sum c_{\alpha,\beta}^\gamma e_\gamma$.

Let $\bigwedge{}^{\bullet}(\g_+)$ be the charged fermion vertex superalgebra generated by $\psi_{\alpha}, \psi_{\alpha}^*$ ($\alpha \in S_+$) which are of opposite parity to $e_\alpha$ and satisfy the $\lambda$-brackets
\begin{align*}
    \lm{\psi_{\alpha}}{\psi_{\beta}^*}=\delta_{\alpha,\beta},\quad \lm{\psi_{\alpha}}{\psi_{\beta}}=\lm{\psi_{\alpha}^*}{\varphi_{\beta}^*}=0.
\end{align*}
Then, we associate the BRST complex 
\begin{align}\label{BRST complex}
C^\bullet_f(V^\kk(\g)):=V^\kk(\g)\otimes \bigwedge{}^{\bullet}(\g_+)
\end{align}
equipped with the differential $d=\int Y(Q,z) \mathrm{d}z$ where
\begin{align*}
    Q=\sum_{\alpha \in S_+} (-1)^{p({e}_\alpha)}(e_\alpha+(f|e_\alpha)) \psi_\alpha^*-\frac{1}{2} \sum_{\alpha,\beta,\gamma \in S_+} (-1)^{p({e}_\alpha)p({e}_\gamma)} c_{\alpha\beta}^\gamma \psi_\gamma \psi_\alpha^* \varphi_\beta^*
\end{align*}
where $(\cdot|\cdot)$ denote the normalized invariant bilinear form on $\g$.

By \cite{KW04}, one has the cohomology vanishing $\HH^{\neq0}(C_{f}^\bullet(V^\kk(\g)))=0$ and the vertex superalgebra obtained at the zero-th cohomology is called the $\W$-superalgebra associated with $(\g,f,\Gamma)$ at level $\kappa$. Since it only depends on the orbit $\OO$ containing $f$ and usually $\Gamma$ is clear, we will denote it as 
\begin{align*}
    \W^\kk(\g,\OO)=\HH^0(C_{f}^\bullet(V^\kk(\g))).
\end{align*}
Note that by replacing $V^\kk(\g)$ in \eqref{BRST complex} with $V^\kk(\g)$-modules $M$, we obtain $\W^\kk(\g,\OO)$-modules, giving rise to a functor
$$\HH_\OO(-)\colon V^\kk(\g)\Mod \rightarrow \W^\kk(\g,\OO)\Mod,\quad M \mapsto \HH^{0}(C_{f}^\bullet(M)).$$

\section{\tmath{$\W$}-superalgebra \tmath{$\W^{\kk}(\gl_{n|1})$} and Heisenberg coset}\label{Wsalg and Heisenberg coset}
\subsection{Regular \tmath{$\W$}-superalgebra} \label{sec: reg Wsalg}
Let $\g=\gl_{n|1}$ and $\g=\nil^+ \oplus \h \oplus \nil^-$ be the standard triangular decomposition. 
Then $\h$ is spanned by the $e_{1,1},\cdots, e_{n+1,n+1}$, $\nil^+$ is generated by $e_{1,2},\cdots,e_{n,n+1}$ which are root vectors corresponding to simple roots $\alpha_1,\cdots, \alpha_{n}$. 
We will identify $\h$ with its dual $\h^*$ by the bilinear form $\kk_V$ and thus 
$$\alpha_1=e_{1,1}-e_{2,2},\cdots, \alpha_{n-1}=e_{n-1,n-1}-e_{n,n},\quad \alpha_{n}=e_{n,n}+e_{n+1,n+1},$$
which are naturally identified with the coroots $h_i$, i.e., $h_i=\alpha_i$ for $i=1,\cdots,n$.

Let $n>1$ and take the regular nilpotent element
$$f=e_{2,1}+e_{3,2}+\cdots +e_{n,n-1}$$
in $\gl_{n|1}$ (resp. $\sll_{n|1}$)
and the good grading $\Gamma$ given by the weighted Dynkin diagram
\begin{center} 
\begin{tikzpicture}
{\dynkin[root
radius=.1cm,labels={\alpha_1,\alpha_2,\alpha_{n-1},\alpha_n},labels*={1,1,1,0},edge length=1cm]A{oo.ot}};
\end{tikzpicture}
\end{center}
and denote by $\W^\kk(\gl_{n|1})$ (resp. $\W^\kk(\sll_{n|1})$) the corresponding $\W$-superalgebra.

Let $\mathcal{A}_{\psi}$ denote the $bc$-system vertex superalgebra generated by odd fields $\psi, \psi^*$ satisfying the $\lambda$-brackets
\begin{align*}
\lm{\psi}{\psi^*}=1,\quad \lm{\psi}{\psi}=\lm{\psi^*}{\psi^*}=0.
\end{align*}
and $\pi_\h^{k+\dC}$ the Heisenberg vertex algebra associated with the Cartan subalgebra $\h$. By definition, it is generated by the fields $u_1,\cdots,u_{n+1}$ which correspond to $e_{1,1},\cdots, e_{n+1,n+1}$ and satisfy the $\lambda$-brackets
$$\lm{u_i}{u_j}=(k+\dC)\kk_V(u_i,u_j)\lambda.$$
For $\lambda \in \h^*$, let $\pi^{k+\dC}_{\h,\lambda}$ denote the Fock module over $\pi^{k+\dC}_{\h}$ of highest weight $\lambda$. 
For later convenience, we denote a non-zero highest weight vector by $\hwt{\frac{1}{k+\dC}\lambda}$ and thus it satisfies 
$$\lm{u_i}{\hwt{\frac{1}{k+\dC}\lambda}}=\lambda(u_i)\hwt{\frac{1}{k+\dC}\lambda}.$$
Let us introduce the following screening operators 
\begin{align}
    S_i:=\int Y(P_i \hwt{-\frac{1}{k+\dC}\alpha_i},z)\dz \colon \mathcal{A}_{\psi}\otimes \pi^{k+\dC}_{\h} \rightarrow \mathcal{A}_{\psi}\otimes \pi^{k+\dC}_{\h,-\alpha_i}
\end{align}
with $P_i=\wun$ for $i=1,\cdots,n-1$ and $P_{n}=\psi$ for $i=n$.
By \cite{CGN21}, $\W^\kk(\gl_{n|1})$ is realized as follows.
\begin{theorem}[\cite{CGN21}]\label{thm: Wakimoto for sW}
There is an embedding of vertex superalgebras
\begin{align*}
    \Psi\colon \W^\kk(\gl_{n|1}) \hookrightarrow \mathcal{A}_{\psi}\otimes \pi^{k+\dC}_{\h}
\end{align*}
which depends continuously on $\kk$ so that the generic image is characterized as the joint kernel of the screening operators
\begin{align*}
    \W^\kk(\gl_{n|1}) \overset{\Psi}{\simeq} \bigcap_{i=1}^n \Ker S_i \subset \mathcal{A}_{\psi}\otimes \pi^{k+\dC}_{\h}.
\end{align*}    
\end{theorem}
We note that the decomposition \ref{eq: decomp of gl} naturally induces 
$$\W^\kk(\gl_{n|1})\simeq \W^\kk(\sll_{n|1})\otimes \pi^{\kk}_\Omega.$$

\subsection{Heisenberg coset}\label{sec: Heisenberg coset}
By the boson-fermion correspondence, the vertex superalgebra $\mathcal{A}_{\psi}$ is isomorphic to the lattice vertex superalgebra 
$V_\Z=\bigoplus_{n \in \Z}\pi_{x,nx}$
where $\Z=\Z x$ is an integral lattice equipped with a bilinear form $(n x,m x)=nm$.
The basis $x$ is identified with the Heisenberg field satisfying the $\lambda$-bracket $\lm{x}{x}=\lambda$ and $\pi_{x,nx}$ is the Fock module over $\pi_x$ with highest weight vector $\hwt{nx}$ satisfying $\lm{x}{\hwt{nx}}=n \hwt{nx}$. The isomorphism is given by
\begin{align}\label{eq: boson-fermion}
    \mathcal{A}_\psi \xrightarrow{\simeq} V_\Z,\quad \psi,\psi^* \mapsto \hwt{x}, \hwt{-x}, 
\end{align}
and maps $\psi\psi^*$ to $x$. For convenience, we will regard them as the same algebra. 

Thanks to the realization in Theorem \ref{thm: Wakimoto for sW}, it is clear that $\W^\kk(\gl_{n|1})$ contains a Heisenberg field 
\begin{align}
    J=x+\varpi_{n}
\end{align}
where $\varpi_{n}=-u_{n+1}$ is the $n$-th fundamental weight of $\gl_{n|1}$.
It satisfies the $\lambda$-bracket
\begin{align*}
    \lm{J}{J}=\left(1-(k+\dC)\right)\lambda.
\end{align*}
Hence, when $k \neq -\dC+1$, $J$ is non-degenerate and thus $\W^\kk(\gl_{n|1})$ is semisimple as a module over the Heisenberg vertex algebra $\pi_J$. Let us denote by
\begin{align}
    C^\kk(\gl_{n|1})=\Com\left(\pi_J,\W^\kk(\gl_{n|1}) \right)\subset \W^\kk(\gl_{n|1})
\end{align}
the Heisenberg coset subalgebra and decompose $\W^\kk(\gl_{n|1})$ into 
\begin{align}\label{eq: Heisenberg decomposition}
    \W^\kk(\gl_{n|1})\simeq \bigoplus_{n\in \Z} C^\kk_n(\gl_{n|1})\otimes \pi_{J,n}
\end{align}
as a module over $C^\kk(\gl_{n|1})\otimes \pi_J$ where $\pi_{J,n}$ is the Fock $\pi_J$-module generated by a highest weight vector $\hwt{n}$ satisfying $\lm{J}{\hwt{n}}=n \hwt{n}$. 

Thanks to the semisimplicity of weight modules over $\pi_J$, it is clear that  
$$C^\kk(\gl_{n|1})\hookrightarrow \Com\left(\pi_J, \mathcal{A}_{\psi}\otimes \pi^{k+\dC}_{\h} \right)\simeq \pi_{\tori}^{k+\dC}$$
where $\pi_{\tori}^{k+\dC}$ is the Heisenberg vertex algebra generated by 
\begin{align}\label{eq: def of t's}
    t_1=u_1,\cdots, t_n=u_n, t_{n+1}=u_{n+1}-(k+\dC)x,
\end{align}
which satisfy the $\lambda$-brackets
\begin{align}\label{eq: OPE for ti's}
    \lm{t_i}{t_j}=0,\quad (i\neq j),\qquad \lm{t_i}{t_i}=\begin{cases}
        (k+\dC)\lambda, & (i\neq n+1),\\
        (k+\dC)(k+\dC-1)\lambda, & (i=n+1).
    \end{cases}
\end{align}
Then, we obtain the following characterization of $C^\kk(\gl_{n|1})$ thanks to Theorem \ref{thm: Wakimoto for sW}. 
\begin{theorem}\label{thm: Wakimoto for sCoset}\hspace{0mm}\\
\textup{(1) (\cite{CGN21})}
For $k \neq -\dC+1$, the embedding $\rho$ induces the embedding
\begin{align*}
    \Psi\colon C^\kk(\gl_{n|1}) \hookrightarrow  \pi^{k+\dC}_{\tori}
\end{align*}
which depends continuously on $\kk$ so that the generic image is characterized as the joint kernel of the screening operators
\begin{align*}
    C^\kk(\gl_{n|1}) \overset{\Psi}{\simeq} \bigcap_{i=1}^n \Ker S_i \subset \pi^{k+\dC}_{\tori}.
\end{align*} 
\textup{(2)} For the critical level $\kk=\kk_c$, i.e., $k=-\dC$, 
\begin{align*}
    C^{\kk_c}(\gl_{n|1})= \cent(\W^{\kk_c}(\gl_{n|1})).
\end{align*}
\end{theorem}
\begin{proof}
    (1) is clear from Theorem \ref{thm: Wakimoto for sW}. 
    (2) By construction, we have $C^{\kk_c}(\gl_{n|1})\supset \cent(\W^{\kk_c}(\gl_{n|1}))$. 
    At the critical level $\kk=\kk_c$, we have
    $$\pi^{h+\dC}_\tori=\pi^{h+\dC}_\h=\cent\left(\mathcal{A}_{\psi}\otimes \pi^{k+\dC}_{\h} \right)$$ 
    by \eqref{eq: def of t's}. Hence, the embedding in Theorem \ref{thm: Wakimoto for sW} implies $C^{\kk_c}(\gl_{n|1})\subset \cent(\W^{\kk_c}(\gl_{n|1}))$ and thus $C^{\kk_c}(\gl_{n|1})= \cent(\W^{\kk_c}(\gl_{n|1}))$.
\end{proof}
We note that the boson-fermion correspondence \eqref{eq: boson-fermion} implies that 
\begin{align}
   S_i=\begin{cases}\displaystyle{\int Y(\hwt{-\frac{1}{k+\dC}(t_i-t_{i+1})},z)\dz},&  (i=1,\cdots,n-1),\\
                     \displaystyle{\int Y(\hwt{-\frac{1}{k+\dC}(t_n+t_{n+1})},z)\dz},& (i=n).
   \end{cases}
\end{align}

\subsection{Case of \tmath{$\gl_{1|1}$}}
For $\g=\gl_{1|1}$, the $\W$-superalgebra $\W^\kk(\g)$ is just the affine vertex superalgebra $V^\kk(\gl_{1|1})$ and the free field realization is given by 
\begin{align}
\label{eq: Wakimoto for gl11}
\begin{array}{ccccc}
  \Psi&\colon&  V^\kk(\gl_{1|1})&\hookrightarrow &\mathcal{A}_\psi \otimes \pi^k_\h\\
    &&e_{1,1},\ e_{1,2} & \mapsto & u_1+x,\ \psi,\\
    &&e_{2,1},\ e_{2,2} & \mapsto &\psi^*(u_1+u_2)+k \partial \psi^*,\ u_2-x,
    \end{array}
\end{align}
whose image at $k\neq0$ agrees with the kernel of the screening operator
$$S=\int Y(\psi \mathrm{e}^{-\frac{1}{k}(u+v)},z)\dz.$$
Hence, the Heisenberg coset
$$C^\kk(\gl_{1|1})=\Com(\pi^\kk_{e_{2,2}}, V^\kk(\gl_{1|1}))$$
embeds into the Heisenberg vertex algebra
$$\Com (\pi_{e_{2,2}},\mathcal{A}_\psi \otimes \pi^k_\h)\simeq \pi_{t_1,t_2}$$
where 
$$t_1=u_1,\qquad t_2=u_2-k x$$
whose image at level $\kk\neq 0,1$ can be described by the kernel of the screening operator $S$, which is rewritten as
$$S=\int \mathrm{e}^{-\frac{1}{k}(t_1+t_2)}(z)dz.$$
From this expression, we can relate strong generators to the coefficients 
\begin{align}\label{eq: generators Wn}
    W_n=\binom{\varepsilon}{n-1}D_1 D_2^{n-1} \wun+\binom{\varepsilon}{n}D_2^n \wun,\qquad \varepsilon=\frac{1}{-1+k},
\end{align}
of the expansion of the pseudo-differential operator:
$$\mathcal{L}_k=D_1 D_2^{\varepsilon}=\left(\wun+\sum_{n=1}^\infty W_n \partial^{-n}\right)\partial^{1+\varepsilon} $$
where 
$$D_1=\partial+\varepsilon t_1,\quad D_2=\partial-\varepsilon t_2.$$
\begin{proposition}\label{prop: gen of gl11 coset}
For $k\neq1$, the elements $W_n$ ($n \geq 0$) lie in the coset $C^\kk(\gl_{1|1})$.
\end{proposition}
\begin{proof}
Since $W_n$ is defined for all $\kk\in \C$, it suffices to show it for $k \neq 1$ by continuity, and thus $S\ W_n=0$. Since $W_0=\wun$, we will show $S\ W_{n+1}=0$ for $n \geq 0$.
Let us denote 
$$Q=\mathrm{e}^{-\frac{1}{k}(t_1+t_2)}=\mathrm{e}^{\rH},\qquad \rH=-\frac{1}{k}(t_1+t_2).$$
By using the identities  
$$[Q_{(m)},D_1]=(m+\varepsilon)Q_{(m-1)},\qquad [Q_{(m)},D_2]=(m+1)Q_{(m-1)}$$
we obtain 
\begin{align*}
&Q_{(0)} D_1=D_1 Q_{(0)}+\varepsilon Q_{(-1)},\\
&D_1=D_2-(\varepsilon+1)\rH_{(-1)},
\end{align*}
and 
\begin{align}
   \label{eq: An} A_n=&Q_{(0)} D_2^n=\sum_{p=0}^{n-1} (-1)^{p+1}(p+1) \binom{n}{p+1}D_2^{n-1-p}\partial^p Q,\\
   \label{eq: Bn} B_{n}=&Q_{(-1)} D_2^n=\sum_{p=0}^{n} (-1)^{p}(p+1) \binom{n}{p} D_2^{n-p}\partial^p Q,\\
   \label{eq: Cn} C_{n}=&\rH_{(-1)}D_2^n=\sum_{p=0}^n(-1)^p(p+1)\binom{n}{p} D_2^{n-p}\partial^p \rH.
\end{align}
by induction.
It follows from these formulas that
\begin{align}
   \nonumber S\ W_{n+1} &= Q_{(0)}W_{n+1}\\ 
   \nonumber &=\binom{\varepsilon}{n}Q_{(0)}D_1D_2^{n}\wun+\binom{\varepsilon}{n+1}Q_{(0)}D_2^{n+1} \wun\\
   \nonumber &=\binom{\varepsilon}{n}[\underbrace{Q_{(0)},D_1]}_{\varepsilon Q_{(-1)}}D_2^{n}\wun+\binom{\varepsilon}{n}D_1 Q_{(0)} D_2^{n}\wun+\binom{\varepsilon}{n+1}Q_{(0)}D_2^{n+1} \wun\\
   \label{eq: SW intermediate} &=\varepsilon \binom{\varepsilon}{n} B_n
    + \binom{\varepsilon}{n}(D_2-(\varepsilon+1)\rH_{(-1)})A_n 
    +\binom{\varepsilon}{n+1} A_{n+1}.
\end{align}
By using \eqref{eq: An} and \eqref{eq: Cn} and the identity 
$$(p+1)\binom{n}{p+1}\binom{n-1-p}{q}=n \binom{n-1}{p+q}\binom{p+q}{p},$$
we rewrite the third term $\rH_{(-1)} A_n$ as 
\begin{align}
   \nonumber \rH_{(-1)} A_n&=\sum_{p=0}^{n}(-1)^{p+1}(p+1)\binom{n}{p+1}\\
   \nonumber &\hspace{2cm}\sum_{q=0}^{n-1-p-q}(-1)^q\binom{n-1-p}{q}D_2^{n-1-p-q}(\partial^q H)(\partial^p Q)\\
   \nonumber &=\sum_{p=0}^{n}(-1)^{p+q+1}n\binom{n-1}{p+q}\sum_{q=0}^{n-1-p-q}\binom{p+q}{p}D_2^{n-1-p-q}(\partial^q H)(\partial^p Q)\\
    %&=\binom{\varepsilon}{n}\sum_{p=0}^{n}(-1)^{p+1}n\binom{n-1}{p}D_2^{n-1-p}\partial^p (H Q)\\
   \label{eq: SW intermediate 2}  &=\sum_{p=1}^{n}(-1)^{p}n\binom{n-1}{p-1}D_2^{n-p}\partial^p Q.
\end{align}
Combining \eqref{eq: SW intermediate} and \eqref{eq: SW intermediate 2}, we obtain 
\begin{align*}
    S\ W_{n+1}=\sum_{p=0}^n(-1)^p a_p D_2^{n-p} \partial^p Q
\end{align*}
where 
\begin{align*}
    a_p=&\varepsilon \binom{\varepsilon}{n}(p+1) \binom{n}{p}-\binom{\varepsilon}{n}(p+1) \binom{n}{p+1}\\
    &-\delta_{p\neq 1}(\varepsilon+1)\binom{\varepsilon}{n} n \binom{n-1}{p-1}-\binom{\varepsilon}{n+1}(p+1) \binom{n+1}{p+1}.
\end{align*}
We show $a_p=0$ for all $p$. 
The case $p=0$ follows directly from the identity 
$$\binom{\alpha}{n}=\binom{\alpha-1}{n}+\binom{\alpha-1}{n-1}.$$
The case $p>0$ follows from 
\begin{align*}
    a_p&=\binom{\varepsilon}{n} \Big(\varepsilon(p+1)\binom{n}{p}-(p+1) \binom{n}{p+1}\\
    &\hspace{2cm}-(\varepsilon+1)n \binom{n-1}{p-1}-(\varepsilon-n)\frac{p+1}{n+1}\binom{n+1}{p+1} \Big)\\
    &=\binom{\varepsilon}{n} \left(\varepsilon(p+1)\binom{n}{p}-(p+1) \binom{n}{p+1}-(\varepsilon+1)p \binom{n}{p}-(\varepsilon-n)\binom{n}{p} \right)\\
    &=\binom{\varepsilon}{n} \left(-(p+1)\binom{n}{p+1}+(n-p)\binom{n}{p} \right)\\
    &=0.
\end{align*}
Thus, we obtain $S\ W_{n+1}=0$ for all $n\geq0$.
\end{proof}

Now, we are ready to show the following main result in the case $\g=\gl_{1|1}$.
\begin{theorem}\label{thm: pdiff operator}
For generic $\kk\notin \Q$ or $\kk=\kk_c$, the elements $W_n$ ($n \geq 1$) strongly generate the coset $C^\kk(\gl_{1|1})$.
In particular, the center $\cent(V^{\kk_c}(\gl_{1|1}))=C^{\kk_c}(\gl_{1|1}) $ is strongly generated by the coefficients $W_n$ of the pseudo-differential operator
$$(\partial-u_1)(\partial+u_2)^{-1}=\wun+ \sum_{n=1}^\infty W_n \partial^{-n}.$$
\end{theorem}

\begin{proof}
We first show the assertion for $\kk=\kk_c$, i.e., $k=0$.
We consider the PBW filtration on $V^{\kk}(\gl_{1|1})$, and the filtration on $\mathcal{A}_\psi \otimes \pi^k_\h$ induced from the filtration of the differential operators on $\mathcal{A}_\psi$ and the PBW filtration on $\pi^k_\h$. 
Then, it is clear that the embedding $\Psi$ in \eqref{eq: Wakimoto for gl11} preserves these filtrations and thus induces an embedding of their associated graded Poisson vertex superalgebras 
\begin{align*}
   \overline{\Psi}\colon \gr V^\kk(\gl_{1|1}) \hookrightarrow \gr \mathcal{A}_\psi \otimes \gr \pi^k_\h,
\end{align*}
and thus 
$\overline{\Psi}\colon \gr C^\kk(\gl_{1|1}) \hookrightarrow \pi_{t_1,t_2}$.
By \eqref{eq: generators Wn}, it is clear that the image $\overline{W}_n$ of $W_n$ in the associated graded space is
$$\overline{W}_n=(-\varepsilon)^n \left(\binom{\varepsilon}{n-1}t_1-\binom{\varepsilon}{n}t_2\right)t_2^{n-1},$$
which are obtained in the expansion 
\begin{align*}
    (z+\varepsilon t_1)(z-\varepsilon t_2)^\varepsilon=\left(\wun +\sum_{n=1}^\infty \overline{W}_nz^{-n}\right) z^{1+\varepsilon}
\end{align*}
By setting $\kk=\kk_c$, we have  
$$\overline{W}_n=(-1)^n(t_1+t_2) t_2^{n-1}=(-1)^n(u_1+u_2) v^{n-1}.$$
Hence, they are supersymmetric polynomials associated with the superspace $\C^{1|1}$ (see \S \ref{sec: supersymmetric polynomials} below for details) and thus generate the affine supersymmetric polynomials $\Lambda_{1|1}^{\aff}$ as a differential algebra, i.e., 
$$\langle \overline{W}_n \mid n\geq 0\rangle_{\partial\text{-alg}} \simeq \Lambda_{1|1}^{\aff} \subset \gr C^\kk(\gl_{1|1}).$$
Since $\Lambda_{1|1}^{\aff} \simeq \gr C^\kk(\gl_{1|1})$ by \cite{MM}, $\overline{W}_n$ $(n\geq0)$ generate $\gr C^\kk(\gl_{1|1})$. Thus, $W_n$'s generate $C^\kk(\gl_{1|1})$.
We have $t_1=u_2$, $t_2=u_2$ at $\kk=\kk_c$ and thus $C^{\kk_c}(\gl_{1|1})\subset \pi^0_{\h}=\cent(\mathcal{A}_\psi \otimes \pi^0_\h)$, which implies 
$$\cent(V^{\kk_c}(\gl_{1|1}))\supset C^{\kk_c}(\gl_{1|1})$$
and thus $\cent(V^{\kk_c}(\gl_{1|1}))= C^{\kk_c}(\gl_{1|1})$.
This completes the assertions for $\kk=\kk_c$.

Next, we show the assertion for $\kk\notin \Q$. 
The elements
\begin{align*}
    &\widetilde{W}_1=W_1,\\
    &\widetilde{W}_2=-\frac{1}{\varepsilon(\varepsilon+1)}\left(W_2-\frac{\varepsilon}{2(\varepsilon+1)}W_1^2-\frac{\varepsilon}{2}\partial W_1\right),\\
    &\widetilde{W}_3=W_3+(\varepsilon-1)\Big(\frac{1}{2}\varepsilon(\varepsilon+1)\partial W_2-\frac{\varepsilon}{6}\partial^2 W_1+\varepsilon W_2W_1\\
    &\hspace{4cm}-\frac{\varepsilon}{2(\varepsilon+1)}W_1 \partial W_1-\frac{\varepsilon}{6(\varepsilon+1)^2}W_1^3\Big)
\end{align*}
satisfy the $\lambda$-brackets
\begin{align*}
    &[\widetilde{W}_1{}_\lambda \widetilde{W}_1]=\varepsilon(\varepsilon+1)^2 \lambda,\quad   [\widetilde{W}_1{}_\lambda \widetilde{W}_2]= [\widetilde{W}_1{}_\lambda \widetilde{W}_3]=0,\\
    &[ \widetilde{W}_2 {}_\lambda \widetilde{W}_2]=\partial \widetilde{W}_2+2 \widetilde{W}_2 \lambda-\lambda^3,\\
    &[ \widetilde{W}_2{}_\lambda \widetilde{W}_3]=\partial \widetilde{W}_3+3 \widetilde{W}_3.
\end{align*}
Hence, the coset subalgebra $C^\kk(\gl_{1|1})$ is strongly generated by  
$$\widetilde{W}_1,\ \widetilde{W}_2,\ \widetilde{W}_3,\  \widetilde{W}_{n+1}=\widetilde{W}_{3(1)}\widetilde{W}_n\ (n\geq3)$$
by \cite{CL4}. Since $\kk \notin \Q$, $\overline{W}_n \neq0$.
Then  it follows from induction that 
$$\overline{W}_{3(1)}\overline{W}_{n}=a_{n}\overline{W}_{n+1}+f_n(\overline{W}_{1},\cdots,\overline{W}_{n})$$
for some $a_n\in \Q(k)$ and $f_n(z_1,\cdots,z_n)$ in $\Q(k)[z_1,\cdots,z_n]$. 
It follows by induction that  
$${W}_n=a_n \widetilde{W}_n+F_n(\widetilde{W_1},\cdots,\widetilde{W}_{n-1})$$
for some differential polynomial $F_n(z_1,\cdots,z_{n-1})$.
Thus ${W}_n$ ($n\geq1$) strongly generate $C^\kk(\gl_{1|1})$ as desired. 
This completes the proof.
\end{proof}
\begin{remark}
\textup{By \cite{CL4}, $\widetilde{W}_1$, $\widetilde{W}_2$, $\widetilde{W}_3$ already strongly generate $C^\kk(\gl_{1|1})$.} 
\end{remark}

\subsection{General case}\label{sec: general coset}
To describe the Heisenberg coset $C^\kk(\gl_{n|1})$ in general, we introduce the following pseudo-differential operator and the coefficients in the expansion:
\begin{align}\label{eq: gln1 ps diffop}
    \mathcal{L}_k=D_1\cdots D_n D_{n+1}^{\varepsilon}=\left(\wun+\sum_{p=1}^\infty W_p \partial^{-p}\right)\partial^{n+\varepsilon} 
\end{align}
where 
$$D_1=\partial+\varepsilon t_1,\cdots, D_n=\partial+\varepsilon t_n,\ D_{n+1}=\partial-\varepsilon t_{n+1},\qquad \varepsilon=\frac{1}{-1+(k+\dC)}.$$
\begin{theorem}\label{prop: coset for Wsalg}
For $k\neq-\dC+1$, the elements $W_p$ $(p \geq 1)$ lie in the coset $C^\kk(\gl_{n|1})$. Moreover, they strongly generate $C^\kk(\gl_{n|1})$ for generic $\kk$.
\end{theorem}
\begin{proof}
By Theorem \ref{thm: Wakimoto for sCoset}, it suffices to show $S_i W_p=0$ for $i=1,\cdots n$ and all $p\geq1$ at generic levels $k$, which is equivalent to $[S_i,\mathcal{L}]=0$ thanks to $[S_i,\partial]=0$.
Since we have $[S_i,D_j]=0$ for $j\neq i,i+1$ and 
\begin{align*}
    [S_i,\mathcal{L}]&=[S_i,D_1]D_2\cdots D_n D_{n+1}^\varepsilon+D_1[S_i,D_2]\cdots D_n D_{n+1}^\varepsilon\\
     &\hspace{2cm}+\cdots+ D_1D_2\cdots D_n [S_i,D_{n+1}^\varepsilon],
\end{align*}
we have
\begin{align*}
    [S_i,\mathcal{L}]=
    \begin{cases} D_1\cdots D_{i-1}[S_i,D_i D_{i+1}]D_{i+2}\cdots D_n D_{n+1}^\varepsilon, &(i\neq n),\\
    D_1\cdots D_{n-1} [S_i,D_n D_{n+1}^\varepsilon], &(i= n).
    \end{cases} 
\end{align*}
For $i\neq n$, $[S_i,D_i D_{i+1}]=0$ is equivalent to the coefficients
$$D_i D_{i+1}=\partial^2+\varepsilon(t_i+t_{i+1})\partial +\varepsilon(\varepsilon t_i t_{i+1}+\partial t_{i+1})$$
are elements of the $\W$-algebra $\W^{\kk'}(\gl_2)$ with $k'+2=k+\dC$, which is well known, see for example \cite{AM, FL}. Similarly, the case $i=n$ follows from Proposition \ref{prop: gen of gl11 coset}. 

Next, we show the assertion for $\kk\notin \Q$. 
The elements
\begin{align*}
    &\widetilde{W}_1=W_1,\\
    &\widetilde{W}_2=-\frac{1}{\varepsilon(\varepsilon+1)}\left(W_2-\frac{\varepsilon+n-1}{2(\varepsilon+n)}W_1^2-\frac{\varepsilon+n-1}{2}\partial W_1\right),\\
    &\widetilde{W}_3=W_3+(\varepsilon+n-2)\Big(\frac{1}{2}\varepsilon(\varepsilon+1)\partial W_2-\frac{\varepsilon+n-1}{6}\partial^2 W_1+\frac{\varepsilon(\varepsilon+1)}{\varepsilon+n}W_2W_1\\
    &\hspace{3cm}-\frac{\varepsilon+n-1}{2(\varepsilon+n)}W_1 \partial W_1-\frac{\varepsilon+n-1}{6(\varepsilon+n)^2}W_1^3\Big)
\end{align*}
satisfy the $\lambda$-brackets
\begin{align*}
    &[\widetilde{W}_1{}_\lambda \widetilde{W}_1]=\varepsilon(\varepsilon +1)(\varepsilon +n) \lambda,\quad   [\widetilde{W}_1{}_\lambda \widetilde{W}_2]= [\widetilde{W}_1{}_\lambda \widetilde{W}_3]=0,\\
    &[ \widetilde{W}_2 {}_\lambda \widetilde{W}_2]=\partial \widetilde{W}_2+2 \widetilde{W}_2 \lambda-\frac{c_n(\varepsilon)}{2}\lambda^3,\\
    &[ \widetilde{W}_2{}_\lambda \widetilde{W}_3]=\partial \widetilde{W}_3+3 \widetilde{W}_3,
\end{align*}
where 
$$c_n(\varepsilon)=-\frac{n(\varepsilon+n-1)(2\varepsilon+n+1)}{\varepsilon(\varepsilon+1)}.$$
Hence, the coset subalgebra $C^\kk(\gl_{n|1})$ is strongly generated by  
$$\widetilde{W}_1,\ \widetilde{W}_2,\ \widetilde{W}_3,\  \widetilde{W}_{n+1}=\widetilde{W}_{3(1)}\widetilde{W}_n\ (n\geq3)$$
by \cite{CL4}. Following the same inductive argument as in the proof of Theorem \ref{thm: pdiff operator}, we obtain that ${W}_n$ ($n\geq1$) strongly generate $C^\kk(\gl_{n|1})$ as desired. 
This completes the proof.
\end{proof}
In contrast to the case $\g=\gl_{1|1}$, the strong generating property at the critical level for the general case $\g=\gl_{n|1}$ in Theorem \ref{prop: coset for Wsalg} is much more subtle and we will prove it in \S \ref{sec: Center} below.

Recall that given a vertex superalgebra $V$, Li filtration $F^\bullet V$ is a decreasing filtration defined by
\begin{align*}
    F^p V=\left\{a^1_{(-n_1-1)}\cdots a^r_{(-n_r-1)}b; \begin{array}{l}
      a^1,\dots,a^r, b \in V  \\
      n_i \geq0,\ n_1+\cdots +n_r \geq p
    \end{array} \right\}.
\end{align*}
The associated graded vector (super)space $\gr V=\oplus_p F^p V/F^{p+1}V$ has a structure of Poisson vertex (super)algebra, see \cite{Ar12, Li}. In particular, the quotient $F^0 V/F^{1}V$, denoted by $R_{V}$ has a structure of Poisson superalgebra, called the $C_2$-algebra.

Moreover, given a homomorphism $f\colon V\rightarrow W$ of vertex superalgebras, the above association induces a homomorphism $\overline{f}\colon \gr V\rightarrow \gr W$ of Poisson vertex superalgebras. We apply this construction to $\Psi$ in Theorem \ref{thm: Wakimoto for sW} and obtain an embedding 
\begin{align*}
    \overline{\Psi}\colon \gr \W^\kk(\gl_{n|1}) \hookrightarrow \gr \mathcal{A}_{\psi}\otimes \gr \pi^{k+\dC}_{\h}.
\end{align*}
Note that 
$$R_{\mathcal{A}_{\psi}\otimes \pi^{k+\dC}_{\h}}\simeq R_{\mathcal{A}_{\psi}}\otimes R_{\pi^{k+\dC}_{\h}}\simeq \bigwedge{}_{\overline{\psi},\overline{\psi}^*} \otimes S(\h)$$
where $\bigwedge{}_{\overline{\psi},\overline{\psi}^*}$ is the exterior algebra generated by  $\overline{\psi},\overline{\psi}^*$ and $S(\h)$ is the symmetric algebra associated with $\h$.

The following will be important later to construct the generators of $\W^\kk(\gl_{n|1})$. 
\begin{proposition}\label{prop: algebraic independence in cosets}
For $k\neq-\dC+1$, the images $\overline{W}_1,\cdots,\overline{W}_n$ in $\bigwedge{}_{\overline{\psi},\overline{\psi}^*} \otimes S(\h)$ are algebraically independent.
\end{proposition}
\proof
It suffices to show the algebraic independence after the canonical quotient
$$\pi\colon \bigwedge{}_{\overline{\psi},\overline{\psi}^*} \otimes S(\h)\twoheadrightarrow S(\h) \twoheadrightarrow \C[e_{1,1},\cdots e_{n,n}].$$ 
It is clear from \eqref{eq: def of t's} and \eqref{eq: gln1 ps diffop} that 
\begin{align*}
    (z+\varepsilon e_{1,1})\cdots (z+\varepsilon e_{n,n})=z^n+ \pi(\overline{W}_1)z^{n-1}+\cdots + \pi(\overline{W}_n)
\end{align*}
where $z$ is an indeterminate. Thus, the algebraic independence among the elementary symmetric polynomials implies that among   $\pi(\overline{W}_i)$'s.

\endproof
\section{Duality between \tmath{$\W^{\kk}(\sll_{n|1})$} and \tmath{$\subW{\op{\kk}}$}}\label{sec: duality and strong generators}
In order to describe the structure of $\W^{\kk}(\gl_{n|1})$, we will use a duality \cite{CGN21,FS} to reduce it to its dual, namely the subregular $\W$-algebra $\W^{\check{\kk}}(\gl_{n},\OO_\sub)$.

\subsection{Subregular \tmath{$\W$}-algebras}\label{sec: Subregular Walg}
Let $\g=\gl_{n}$ ($n>1$) and $\g=\nil^+ \oplus \h \oplus \nil^-$ be the standard triangular decomposition. 
Then $\h$ is spanned by the $e_{1,1},\cdots, e_{n,n}$, $\nil^+$ is generated by $e_{1,2},\cdots,e_{n-1,n}$ which are root vectors corresponding to simple roots $\alpha_1,\cdots, \alpha_{n-1}$. 
We will identify $\h$ with its dual $\h^*$ by the bilinear form $\kk_V$ and thus 
$$\alpha_1=e_{1,1}-e_{2,2},\cdots, \alpha_{n-1}=e_{n-1,n-1}-e_{n,n},$$
which are naturally identified with the coroots $h_i$, i.e., $h_i=\alpha_i$ for $i=1,\cdots,n$.

Take the subregular nilpotent element
\begin{align}\label{eq: subregular nilp}
    f=e_{2,1}+e_{3,2}+\cdots +e_{n,n-1}
\end{align}
in $\gl_{n}$ (resp. $\sll_{n}$)
and the good grading $\Gamma$ given by the weighted Dynkin diagram
\begin{center} 
\begin{tikzpicture}
{\dynkin[root
radius=.1cm,labels={\alpha_1,\alpha_{n-2},\alpha_{n-1}},labels*={1,1,0},edge length=1cm]A{o.oo}};
\end{tikzpicture}
\end{center}
and denote by $\W^{{\kk}}(\gl_{n},\OO_\sub)$ (resp. $\W^{{\kk}}(\sll_{n},\OO_\sub)$) the corresponding $\W$-algebra. 
It is called the subregular $\W$-algebra at level ${\kk}$.  

Let $\mathcal{A}$ denote the $\beta\gamma$-system vertex algebra generated by (even) fields $a, a^*$ satisfying the $\lambda$-brackets
\begin{align}\label{eq: betagamma system}
\lm{a}{a^*}=1,\quad \lm{a}{a}=\lm{a^*}{a^*}=0.
\end{align}
The localization $\mathcal{A}^\times$ by taking the inverse $a$ is realized inside the lattice vertex algebra $V_L$
$$\mathcal{A}^\times\simeq \bigoplus_{n\in \Z}\pi^{x,y}_{n(x+y)} \subset V_L$$
associated with the integer lattice $L=\Z x\oplus \Z y$ equipped with a bilinear form $(nx+m y,n'x+m'y)=nn'-mm'$. The embedding $\mathcal{A}\hookrightarrow  V_L$ has a following characterization 
\begin{align*}
    \mathcal{A}\simeq \Ker S_n  \subset \mathcal{A}^\times,\qquad S_n=\int Y(\hwt{x},z)\dz,
\end{align*}
which sends
\begin{align*}
    a \mapsto \hwt{x+y},\quad a^* \mapsto-x \hwt{x+y}.
\end{align*}
Let $\pi_\h^{{k}+{\dC}}$ be the Heisenberg vertex algebra associated with ${\h}\subset {\g}$. By definition, it is generated by the fields $u_1,\cdots,u_{n}$ which correspond to $e_{1,1},\cdots, e_{n,n}$ and satisfy the $\lambda$-brackets
$$\lm{u_i}{u_j}=(k+\dC)\kk_V(u_i,u_j)\lambda.$$

Let us introduce the following screening operators 
\begin{align}
    S_i:=\int Y(P_i \hwt{-\frac{1}{k+\dC}\alpha_i},z)\dz \colon \mathcal{A}^\times\otimes \pi^{k+\dC}_{\h} \rightarrow \mathcal{A}^\times\otimes \pi^{k+\dC}_{\h,-\alpha_i}
\end{align}
for $i=1,\cdots, n-1$ with $P_i=\wun$ for $i=1,\cdots,n-2$ and $P_{n-1}=a$ for $i=n$.
Then, $\W^{{\kk}}(\gl_{n},\OO_\sub)$ is realized in $\mathcal{A}^\times \otimes \pi_\h^{{k}+{\dC}}$ as follows.

\begin{proposition}[\cite{CGN21}]\label{Wakimoto for subreg 2}
There is an embedding of vertex algebras
\begin{align*}
    \Psi\colon \W^\kk(\gl_{n},\OO_\sub) \hookrightarrow \mathcal{A}^\times \otimes \pi^{k+\dC}_{\h}
\end{align*}
which depends continuously on $\kk$ so that the generic image is characterized as the joint kernel of the screening operators
\begin{align*}
    \W^\kk(\gl_{n},\OO_\sub) \overset{\Psi}{\simeq} \bigcap_{i=1}^n \Ker S_i \subset \mathcal{A}^\times\otimes \pi^{k+\dC}_{\h}.
\end{align*}    
\end{proposition}
As in the super-case, we have the decomposition 
\begin{align*}
    \W^\kk(\gl_{n},\OO_\sub)\simeq \W^\kk(\sll_{n},\OO_\sub)\otimes \pi^{\kk}_\Omega.
\end{align*}

Let us recall from \cite{GK} the explicit strong generators of $\W^{\check{\kk}}
(\sll_n,\OO_\sub)$ under this embedding. 
For the convenience in \S \ref{sec: duality} below, we change the notation by adding the symbol ``$\check{\hspace{5mm}}$", for example the level $\kk$ is changed to $\check{\kk}$ and the Heisenberg vertex algebra $\pi^{k+\dC}_\h$ to $\pi^{\check{k}+\check{\dC}}_{\check{\h}}$. 
Let us introduce the following elements:
\begin{align*}
    \check{G}_\tp= \hwt{x+y},\quad \check{H}=-y+\check{\varpi}_{n-1},\quad \check{G}_\tn=\check{R}_1 \check{R}_2 \cdots \check{R}_n \hwt{-(x+y)}
\end{align*}
where we set 
\begin{align*}
    \check{R}_i=\begin{cases}
        \check{\varkappa} \pd -y +(\check{h}_i+\cdots +\check{h}_{n-1}), & (i=1,\dots,n-1),\\
        \check{\varkappa} \pd +\check{\varkappa} (x+y)-x, & (i=n),
    \end{cases}
\end{align*}
with $\check{\varkappa}=(\check{k}+\check{\dC}-1)$ and $\check{\varpi}_{n-1}=\frac{1}{n}\check{h}_1+\frac{2}{n}\check{h}_2+\cdots+\frac{n-1}{n}\check{h}_{n-1}$ is the $(n-1)$-th fundamental weight where $\check{h}_i=\check{u}_i-\check{u}_{i+1}$.
Next, we introduce $\check{W}_1,\cdots, \check{W}_{n-1}$ inductively by the formula 
\begin{align*}
    \check{W}_m=\check{W}_m{}''-\sum_{i=0}^{m-1}(-1)^{m-i}\binom{n-i}{m-i}\check{W}_m (\check{\varkappa} \pd+ \check{H})^{m-i}\wun
\end{align*}
where 
\begin{align*}
    \check{W}_m{}''=\sum_{i=0}^m(-1)^{m+i}\left(\prod_{j=1}^i \frac{j(\check{\varkappa}+1)}{j \check{\varkappa}}\right)e_{m-k}(\check{R}_1, \check{R}_2, \cdots ,\check{R}_n)(\check{\varkappa} \pd+ \check{\varkappa}y)^i\wun
\end{align*}
and $e_{p}(z_1,\cdots,z_n)=\sum_{i_1>\cdots >i_p}z_{i_1}z_{i_2}\cdots z_{i_m}$
is the $p$-th (non-commutative) elementary symmetric polynomial.

\begin{theorem}[\cite{GK}]
For all $\check{\kappa}\in \C$, $\W^{\check{\kk}}(\sll_n,\OO_\sub)$ under the embedding $\check{\Psi}$ is strongly generated by 
$$\check{G}_\tpn, \check{H}, \check{W}_2,\cdots, \check{W}_{n-1}.$$
\end{theorem}

\subsection{Duality}\label{sec: duality}
The $\W$-superalgebras
\begin{align*}
    \subW{\op{\kk}}, \qquad \W^{\kk}(\sll_{n|1}),
\end{align*}
are reconstructed from one to the other, called the Kazama-Suzuki type coset construction \cite{CGN21, FS}.
We recall the reconstruction of $\W^{\kk}(\sll_{n|1})$ in terms of $\subW{\op{\kk}}$ when the levels $(\kk,\check{\kk})$ satisfy the relation
\begin{align}\label{eq: dual levels}
    (k+\dC)(\check{k}+\check{\dC})=1.
\end{align}
By Proposition \ref{Wakimoto for subreg 2}, we have the embedding 
\begin{align*}
    \W^{\check{\kk}}(\sll_{n},\OO_\sub)\otimes V_\Z \overset{\check{\Psi} \otimes 1}{\hookrightarrow} (\pi^{\check{k}+\check{\dC}}_{\check{\h}}\otimes \mathcal{A}^\times) \otimes V_\Z.
\end{align*}
Let $\pi_{H_\Delta}\subset \W^{\check{\kk}}(\sll_{n},\OO_\sub)\otimes V_\Z$ be the Heisenberg vertex subalgebra generated by 
$$H_\Delta= \check{H}\otimes \wun- \wun \otimes x= \check{\varpi}_{n-1}-y_1 -x_2,$$
which satisfies the $\lambda$-bracket 
$$\lm{H_\Delta}{H_\Delta}=\frac{n-1}{n}(\check{k}+\check{\dC})\lambda.$$
Here, the sub-indices $i$ for $x$ and $y$ denote the tensor factor for clarity.
Let $\check{k}\neq -\check{\dC}$ and we consider the Heisenberg coset subalgebra:
\begin{align*}
    \Com(\pi_{H_\Delta}, \W^{\check{\kk}}(\sll_{n},\OO_\sub)\otimes V_\Z) \overset{\check{\Psi} \otimes 1}{\hookrightarrow}  \Com(\pi_{H_\Delta},(\pi^{\check{k}+\check{\dC}}_{\check{\h}}\otimes \mathcal{A}^\times) \otimes V_\Z).
\end{align*}
\begin{lemma}\label{FS duality for FF}
    For levels $(\kk,\check{\kk})$ satisfying \eqref{eq: dual levels}, there is an unique isomorphism 
    $$\sigma \colon  \pi^{k+\dC}_{\h}\otimes V_\Z  \xrightarrow{\simeq}  \Com (\pi_{H_\Delta}, (\pi^{\check{k}+\check{\dC}}_{\check{\h}} \otimes \mathcal{A}^\times) \otimes V_\Z),$$
    which sends 
    \begin{align*}
    &h_i \mapsto \begin{cases}
        -(k+\dC) \check{h}_i, & (i\neq n-1,n),\\
        -(k+\dC)\check{h}_{n-1}+(x_1+y_1), & (i=n-1),\\
        (k+\dC) (y_1+x_2), & (i=n),\\
    \end{cases}\\
    &x \mapsto  x_1+y_1+x_2.
\end{align*}
\end{lemma}
The isomorphism $\sigma$ identifies the screening operators which characterize the Heisenberg coset subalgebra $\Com(\pi_{H_\Delta}, \W^{\check{\kk}}(\sll_{n},\OO_\sub))$ in Proposition \ref{Wakimoto for subreg 2} and those for $\W^{\kk}(\sll_{n|1})$ in Theorem \ref{thm: Wakimoto for sW}.
Thus, we obtain the reconstruction of $\W^{\kk}(\sll_{n|1})$ as the Heisenberg coset subalgebra. 
\begin{theorem}[\cite{CGN21, FS}] \label{FS duality}
    For levels $(\kk,\check{\kk})$ satisfying \eqref{eq: dual levels}, there is an isomorphism of vertex superalgebras 
    \begin{align*}
        \W^{\kk}(\sll_{n|1}) \simeq \Com (\pi_{H_\Delta}, \W^{\check{\kk}}(\sll_n,\OO_\sub) \otimes V_\Z),
    \end{align*}
    which makes the following diagram commutative:
    \begin{center}
\begin{tikzcd}[ column sep = huge]
			\pi^{k+\dC}_{\h}\otimes V_\Z
			\arrow[r, " \overset{\sigma}{\simeq}"]&
			\Com (\pi_{H_\Delta}, (\pi^{\check{k}+\check{\dC}}_{\check{\h}} \otimes \mathcal{A}^\times) \otimes V_\Z)\\
		      \W^{\kk}(\sll_{n|1})
			\arrow[r, " \simeq "] \arrow[u,symbol=\subset, " \Psi\ "]&
			 \Com (\pi_{H_\Delta}, \W^{\check{\kk}}(\sll_n,\OO_\sub) \otimes V_\Z) \arrow[u,symbol=\subset, " \check{\Psi}\ "].
\end{tikzcd}
\end{center}
\end{theorem}

\begin{remark}
\textup{When $n=1$, it is known that $V^\kk(\gl_{1|1})$ at non-critical levels are all isomorphic and the analogue of the realization in Theorem \ref{FS duality} is given by the $U(1)$-orbifold construction
$$V^{\kk}(\gl_{1|1})\simeq (\sA\otimes V_\Z)^{U(1)}.$$
This implies that the Heisenberg coset subalgebra $C^\kk(\gl_{1|1})$ is isomorphic to the orbifold $\sA^{U(1)}$ and thus $W_{1+ \infty, c}$-algebra at central charge $c=-1$ (cf. \cite{CL4, KR}).
}
\end{remark}

\subsection{Generators of \tmath{$\W^\kk(\gl_{n|1})$}}
We are ready to describe the generators of the $\W$-superalgebra $\W^\kk(\gl_{n|1})$ under the realization $\Psi$ in Theorem \ref{thm: Wakimoto for sW}.

By \cite{KW04}, $\W^{\kk}(\gl_{n|1})$ at arbitrary levels $\kk$ is strongly generated by certain fields 
$$w_{\mathfrak{q}_1},\cdots,w_{\mathfrak{q}_d}$$
 associated to an (arbitrary) basis $\mathcal{B}=\{\mathfrak{q}_1,\cdots \mathfrak{q}_d\}$ of the centralizer $\gl_{n|1}^{f}\subset \gl_{n|1}$
and their differential monomials form a PBW base.
Moreover, the (quasi-) conformal weights of $w_{\mathfrak{q}_i}$ are given by $1-\Gamma(\mathfrak{q}_i)$ by using the good grading $\Gamma(\mathfrak{q}_i)$ of $\mathfrak{q}_i$.
By using the explicit forms of $f$ and $\Gamma$ in \S \ref{sec: reg Wsalg}, we can take a basis $\mathcal{B}$ as 
\begin{align*} 
    &\mathfrak{q}_1^{(1)}= e_{1,1}+\cdots +e_{n+1,n+1},\\
    &\mathfrak{q}_1^{(2)}=-e_{n+1,n+1},\\
    &\mathfrak{q}_{i}= e_{i,1}+e_{i+1,2}+\cdots+e_{n,n-i+1},\qquad (i=2,\cdots, n)\\
    &\mathfrak{q}_\tp= e_{n,n+1},\quad  \mathfrak{q}_\tn= e_{n+1,1}.
\end{align*}
The conformal weights of the corresponding fields in $\W^\kk(\gl_{n|1})$ are given by
\begin{align}\label{eq: strong generating type of Wsalg}
    1,1,2,\cdots, n,1,n,
\end{align}
respectively. Indeed, the fields corresponding to $\mathfrak{q}_1^{(1)}, \mathfrak{q}_1^{(2)}$ are Heisenberg fields
\begin{align*}
    \Omega=u_1+\cdots +u_{n+1},\quad J=-u_{n+1}+x
\end{align*}

Next, let us introduce the following odd fields
\begin{align*}
    G_\tp= \hwt{x},\quad G_\tn=R_1^\tn R_2^\tn \cdots R_n^\tn \hwt{-x}
\end{align*}
where
\begin{align*}
    R_i=\varkappa \pd -x +(h_i+\cdots +h_{n}),\quad (i=1,\dots,n)
\end{align*}
with $\varkappa=(k+\dC-1)$.
\begin{theorem}\label{strong generators for reg super}
Under the embedding $\Psi$, $\W^{\kk}(\gl_{n|1})$ is strongly generated by 
$$ \Omega, J, W_2,\cdots, W_{n},G_\tpn$$ 
for $k \neq -\dC+1$.
\end{theorem}
\proof
We first show that $ \Omega, J, W_2,\cdots, W_{n},G_\tpn$ lie in $\W^{\kk}(\gl_{n|1})$.
Since $W_i$'s lie in $\W^{\kk}(\gl_{n|1})$ by Proposition \ref{prop: coset for Wsalg}, it remains to check $G_\tpn$.
It is straightforward to show that $G_\tp$ lies in $\W^{\kk}(\gl_{n|1})$ by checking $S_i G_\tp=0$, thanks to Theorem \ref{thm: Wakimoto for sW}.
Next, we show that $G_\tn$ lies in in $\W^{\kk}(\gl_{n|1})$.
By Theorem \ref{FS duality} again, it suffices to show 
\begin{align}
  \nonumber  G_\tn&=-(-1)^n(\varkappa+1)^{n} \check{G}_\tn \otimes \hwt{-x}\\
  \label{negative generator}   &=-(-1)^n(\varkappa+1)^{n} \check{R}_1 \check{R}_2 \cdots \check{R}_n \hwt{-(x_1+y_1)} \otimes \hwt{-x_2}
\end{align}
where in the second line, we put the indices $x_i$ to clarify the tensor factor.
By using the isomorphism $\sigma$ in Lemma \ref{FS duality for FF}, we have 
\begin{align*}
    \check{R}_n \hwt{-(x_1+y_1)}\otimes \hwt{-x_2}
    &=-x_1 \hwt{-(x_1+y_1+x_2)}
    =\left(\pd+ \frac{1}{\varkappa+1}h_n\right) \hwt{-x}\\
    &=\frac{1}{\varkappa+1}( (\varkappa+1)\pd +h_n) \hwt{-x}
    =\frac{1}{\varkappa+1}(\varkappa\pd +h_n-x) \hwt{-x}\\
    &=\frac{1}{\varkappa+1}R_n \hwt{-x}.
\end{align*}
For $i=1,\dots,n-1$, 
\begin{align*}
   \check{R}_i&=\check{\varkappa} \pd -y_1 +(\check{h}_i+\cdots +\check{h}_{n-1})\\
    &=\check{\varkappa} \pd -y_1-\frac{1}{\varkappa+1}(h_i+\cdots+ h_{n-1})+\frac{1}{\varkappa+1}(x_1+y_1)\\
    &=-\frac{1}{\varkappa+1}\left(\varkappa \pd +(\varkappa+1)y_1+(h_i+\cdots+ h_{n-1})-(x_1+y_1) \right)\\
    &=-\frac{1}{\varkappa+1}\left(\varkappa \pd +(h_i+\cdots+ h_{n-1} + h_n-x)-\varkappa x_2 \right).
\end{align*}
In the expression \eqref{negative generator}, the derivation $\pd$ appearing in $\check{R}_i$ acts only on the first factor, i.e., acts as $\pd \otimes 1$. On the other hand, the total derivative $\pd(=\pd \otimes 1+1 \otimes \pd)$ acts on $\check{R}_i\cdots \check{R}_n \hwt{-(x_1+y_1)}\otimes \hwt{-x_2}$
by $\pd=\pd \otimes 1-1\otimes x_2$. Therefore, we have
\begin{align*}
   \check{R}_i
    &=-\frac{1}{\varkappa+1}\left(\varkappa \pd\otimes 1 +(h_i+\cdots+ h_{n-1} + h_n-x)-\varkappa x_2 \right)\\
    &=-\frac{1}{\varkappa+1}\left(\varkappa (\pd-x_2) +(h_i+\cdots+ h_{n-1} + h_n-x)-\varkappa x_2 \right)\\
    &=-\frac{1}{\varkappa+1}\left(\varkappa \pd +(h_i+\cdots+ h_{n-1} + h_n-x) \right)\\
    &=-\frac{1}{\varkappa+1}R_i
\end{align*}
holds as operators acting on $\check{R}_{i+1}\cdots \check{R}_n \hwt{-(x_1+y_1)}\otimes \hwt{-x_2}$. Hence, by induction, we obtain 
\begin{align*}
    \check{R}_1\cdots \check{R}_n \hwt{-(x_1+y_1)}\otimes \hwt{-x_2}=- \frac{(-1)^n}{(\varkappa+1)^{n}} R_1\cdots R_n \hwt{-x},
\end{align*}
and thus \eqref{negative generator} holds.

Finally, we show that $\Omega, J, W_2,\cdots, W_{n},G_\tpn$ strongly generate $\W^\kk(\gl_{n|1})$. 
Since $G_\tpn$ are odd fields and $J_{(0)}G_\tpn=\pm G_\tpn$, they are algebraically independent. As they have conformal weight $1,n$, they are strong generators corresponding to $\mathfrak{q}_\tpn$. On the Since $W_1=\varkappa(\Omega+(\varepsilon+1)J)$, we may consider $J, W_1, W_2,\cdots, W_{n}$ instead of $\Omega, J, W_2,\cdots, W_{n}$. 
Then Proposition \ref{prop: algebraic independence in cosets} implies that their differential monomials form the same PBW basis as those of the generators corresponding to $\mathfrak{q}_1^{(1)}, \mathfrak{q}_1^{(2)},\mathfrak{q}_2,\cdots, \mathfrak{q}_n$.
This completes the proof.
\endproof

\subsection{Critical level}
Let $\mathbb{A}^{N}$ denote the affine scheme whose regular functions are just a polynomial ring $\C[z_1,\cdots,z_N]$ and $J_\infty\mathbb{A}^{N}$ its jet scheme. The algebra of the regular functions $\C[J_\infty \mathbb{A}^{N}]$ is a commutative vertex algebra (i.e., a differential algebra) and again a polynomial ring 
$$\C[J_\infty \mathbb{A}^{N}]=\C[\partial^pz_1,\cdots,\partial^pz_N\mid p\geq0].$$

\begin{theorem}\label{critical level structure of reg superW}
For the critical level $\kk=\kk_c$, there are isomorphisms 
\begin{align*}
    &\W^{\kk_c}(\gl_{n|1})\simeq \C[J_\infty \mathbb{A}^{n-1}]\otimes V^{\kk_c}(\gl_{1|1}),\\
    &\W^{\kk_c}(\sll_{n|1})\simeq \C[J_\infty \mathbb{A}^{n-2}]\otimes V^{\kk_c}(\gl_{1|1}).    
\end{align*}
\end{theorem}
\proof
It suffices to show the assertion for $\W^{\kk_c}(\gl_{n|1})$. 
By Proposition \ref{prop: coset for Wsalg}, the gererators $W_1,\cdots,W_n$ lie in the center $\cent(\W^{\kk_c}(\gl_{n|1})$ and thus the first $(n-1)$ elements generate a central subalgebra isomorphic to $\C[J_\infty \mathbb{A}^{n-1}]$, i.e., there is an embedding 
\begin{align}\label{eq: first emb}
    \C[J_\infty \mathbb{A}^{n-1}] \hookrightarrow \W^{\kk_c}(\gl_{n|1}),\quad z_i\mapsto W_i\ (i=1,\cdots,n-1).
\end{align}
We will show that $G_\tpn$, $J$ and $W_n$ after necessary modification generate a subalgebra isomorphic to $V^{\kk_c}(\gl_{1|1})$ and thus 
$$\W^{\kk_c}(\gl_{n|1})\simeq \C[J_\infty \mathbb{A}^{n-1}]\otimes V^{\kk_c}(\gl_{1|1}).$$
It is straightforward to show 
\begin{align*}
\lm{J}{J}=\lambda, \quad \lm{J}{G_\tpn}=\pm G_\tpn.
\end{align*}
and 
$$\lm{G_\tp}{G_\tp}=0,\quad \lm{G_\tn}{G_\tn}=0,$$
thanks to the weight reason. 
Next, we compute the $\lambda$-bracket $\lm{G_\tp}{G_\tn}$.
Since we have
\begin{align*}
    R_i=-(\pd+x+H_i),\quad H_i=-(h_i+\cdots +h_n).
\end{align*}
Hence, one has $R_n \hwt{-x}=-H_i\hwt{-x}$. By induction, one can show that 
\begin{align*}
    R_i \cdots R_n\hwt{-x}=(-1)^{n+1-i} \big((\pd+H_i)\cdots (\pd+H_n) \wun\big) \otimes \hwt{-x} \in \pi^{0}_{\h} \otimes V_\Z
\end{align*}
since $H_1,\dots H_n$ are in the commutative vertex algebra $\pi^{0}_{\h}$.
Hence, 
$$G_\tn=R_1 \cdots R_n\hwt{-x}=(-1)^{n} \big((\pd+H_1)\cdots (\pd+H_n) \wun \big)\otimes \hwt{-x}$$
and thus 
\begin{align*}
    \lm{G_\tp}{G_\tn}&=(-1)^n \big((\pd+H_1)\cdots (\pd+H_n) \wun \big)\lm{\hwt{x}}{\hwt{-x}}\\
    &=(-1)^n (\pd+H_1)\cdots (\pd+H_n) \wun.
\end{align*}
Let us set $\mathscr{W}_n:=(-1)^n (\pd+H_1)\cdots (\pd+H_n) \wun$. Then, it is clear that we have an embedding of vertex superalgebras 
\begin{align}\label{eq: subalg 2}
\begin{array}{ccc}
     V^{\kk_c}(\gl_{1|1})&\hookrightarrow &\W^{\kk_c}(\gl_{n|1})\\
     e_{1,1},\ e_{2,2} & \mapsto & J,\ -J+\mathscr{W}_n,\\
     e_{1,2},\ e_{2,1} & \mapsto &G_\tp,\ G_\tn.
    \end{array}
\end{align}
Next, we consider the image $\pi(\overline{\mathscr{W}}_n)$ as in \S \ref{sec: general coset}. 
Since $\overline{H}_i=-(e_{i,i}+e_{n+1,n+1})$, we have $\pi(\overline{H}_i)=-e_{i,i}$
$$\pi(\overline{\mathscr{W}}_n)=e_{1,1}\cdots e_{n,n}=\varepsilon^{-n}\pi(\overline{W}_n).$$
Hence, there exists a differential polynomial $F(z_1,\cdots,z_{n})$ such that 
$$\mathscr{W}_n=\varepsilon^{-n} W_n+F(J,W_1,\cdots,W_{n-1})$$
by Theorem \ref{strong generators for reg super}.
Thus, we may take the strong generators at the critical level $\kk=\kk_c$ as $J, W_1, \cdots, W_{n-1}, \mathscr{W}_n, G_\tpn$.
Thus, the embeddings \eqref{eq: first emb} and \eqref{eq: subalg 2} combine to induce the isomorphism 
\begin{align*}
    \W^{\kk_c}(\sll_{n|1}) \simeq \C[J_\infty \mathbb{A}^{n-2}] \otimes V^{\kk_c}(\gl_{1|1})
\end{align*}
as desired. This completes the proof.
\endproof
\begin{corollary}\label{center of reg sW alg}
    There are isomorphisms of vertex algebras 
    \begin{align*}
    &\cent(\W^{\kk_c}(\gl_{n|1})) \simeq \C[J_\infty \mathbb{A}^{n-1}] \otimes \cent(V^{\kk_c}(\gl_{1|1})),\\
    &\cent(\W^{\kk_c}(\sll_{n|1})) \simeq \C[J_\infty \mathbb{A}^{n-2}] \otimes \cent(V^{\kk_c}(\gl_{1|1})).
    \end{align*}
\end{corollary}
We note that the center $\cent(V^{\kk_c}(\gl_{1|1}))$ is already described in the literature \cite{A, MM}.

\section{Wakimoto realization}\label{sec: Wakimoto realization}
\subsection{Affine case}
In this section, we recall the Wakimoto realization for the affine vertex superalgebra $V^{\kk}(\g)$ for $\g=\gl_{n|1}$ or $\sll_{n|1}$, which is a super-analogue of \cite{F}. The proof is a modification of the arguments in \cite{F} to this setting, and thus we will only state the results. 

Let $G$ denote be quasi-reductive algebraic groups $GL_{n|1}$ (resp.\ $SL_{n|1}$) so that the associated Lie superalgebras $\g=\gl_{n|1}$, (resp.\ $\sll_{n|1}$).
Corresponding to the standard triangular decomposition $\g=\nil_+ \oplus \h \oplus \nil_-$, we denote by $N_+, H, N_-$ the closed subgroups of $G$.
Using the exponential map $\exp:\n_+\overset{\sim}{\longrightarrow}N_+$, one may identify the coordinate rings $\C[N_+]\simeq \C[z_\alpha\mid \alpha\in \Delta_+]$ where $z_\alpha$ is the coordinate of the root vector $e_\alpha\in\n_+$ ($\alpha \in \Delta_+$).
The inverse of the left multiplication of $N_+$ on itself induces a homomorphism $\rho^R:\n_+\to\D(N_+)$ where $\D(N_+)$ is the ring of differential operators on $N_+$. 
Then, one has 
\begin{equation}\label{eq:rho}
  \rho^R:\n_+\to\D(N_+),\quad e_{\alpha}\mapsto e_{\alpha}^R:=\sum_{\alpha\in \Delta_+}P_i^{\alpha}(z)\partial_{z_\alpha} 
\end{equation}
for some polynomials $P_i^{\alpha}(z)$ in $\C[N_+]$.
On the other hand, consider the following right multiplication 
$$N_- \backslash G \times G \rightarrow G.$$
By restricting to the open subset $U=H N_+\subset N_- \backslash G$, we obtain an embedding 
\begin{align}\label{finite Wakimoto}
   \rho\colon U(\g) \hookrightarrow \D(N_+) \otimes U(\h).
\end{align}

Now, let $\sA$ be the chiral differential operators over the affine superspace $N_+$. Explicitly, $\sA$ is the vertex superalgebra, strongly generated by $a_{\alpha}, a_{\alpha}^*$ ($\alpha \in \Delta_+$) which have the same parity as $e_{\alpha}$ in $\g$ and satisfy the $\lambda$-bracket
\begin{align*}
    \lm{a_{\alpha}}{a_{\beta}^*}=\delta_{\alpha,\beta},\quad \lm{a_{\alpha}}{a_{\beta}}=\lm{a_{\alpha}^*}{a_{\beta}^*}=0.
\end{align*}
Let $\heis{\kk-\kk_c}$ be the Heisenberg vertex algebras associated with the Cartan subalgebra $\h\subset \g$ at the level $\kk-\kk_c$. It is, by definition, generated by the subspace $\h$ which satisfies the $\lambda$-bracket
\begin{align*}
    \lm{\w{h}}{\w{h}'}=(\kappa-\kappa_c)(h,h')\lambda,\qquad (h,h' \in \h).
\end{align*}
For $\nu\in \h^*$, we denote by $\pi^{\kk-\kk_c}_{\h,\nu}$ the Fock module over $\pi^{\kk-\kk_c}$ of highest weight $\nu$. 

As a vertex-algebraic upgrade of the morphism $\rho$ in \eqref{finite Wakimoto}, we have the following realization of the affine vertex superalgebra $V^\kk(\g)$.
\begin{proposition}\label{Wakimoto for affine I}
There exists an embedding of vertex superalgebras 
\begin{align*}
\w{\rho} \colon V^{\kk}(\g)\hookrightarrow \sA\otimes \heis{\kk-\kk_c}.
\end{align*}
Moreover, the image is characterized by the kernel of the screening operators at generic levels:
\begin{align*}
V^{\kk}(\g) \simeq \bigcap \Ker S_i \subset \sA \otimes \heis{\kk-\kk_c}
\end{align*}
$$S_i=\int Y(e_{i}^R\mathrm{e}^{-\frac{1}{k+n}\alpha_i},z)\mathrm{d}z\colon \sA\otimes \heis{\kk-\kk_c}\rightarrow \sA\otimes \pi^{\kk-\kk_c}_{\h,-\alpha_i},\quad (i=1,\cdots,n).$$
\end{proposition}
\noindent
By construction, the map $\w{\rho}$ is a chiralization of the embedding in \eqref{finite Wakimoto} in the sense that by applying Zhu's functor one obtains \eqref{finite Wakimoto}.

The realization of $V^\kk(\g)$ in the above proposition is called the \emph{Wakimoto realization} (see e.g.\ \cite{F}).
The homomorphism $\w{\rho}$ induces a family of $V^\kk(\g)$-modules  $$\affWak{\lambda}:=\sA\otimes \pi^{\kk-\kk_c}_{\h,\lambda},\qquad (\lambda\in \h^*)$$
obtained by restriction, called the \emph{Wakimoto representations} over $V^\kk(\g)$.

The homomorphism $\rho^R$ in \eqref{eq:rho} induces an embedding of vertex superalgebras 
\begin{align*}
\rho^R\colon V^0(\nil_+) \hookrightarrow \sA\otimes \pi^{\kk-\kk_c}_\h.
\end{align*}
Note that the two embeddings
\begin{align}\label{left and right centralize}
\w{\rho} \colon V^0(\nil_+)\hookrightarrow \sA\otimes \pi^{\kk-\kk_c}_\h \hookleftarrow V^0(\nil_+) \colon \rho^R
\end{align} 
centralize each other.

\begin{lemma}\label{FF center goes Heis}
At the critical level $\kk=\kk_c$, we have 
\begin{align*}
\w{\rho}\colon\cent(V^{\kk_c}(\g)) \hookrightarrow \pi^{0}_\h \subset \mathcal{A}\otimes \pi^{0}_\h.
\end{align*}
\end{lemma}
\proof
One can show the assertion by adapting, to our setting, the proof of \cite[Lemma 7.1]{F} for the simple Lie algebras.
Observe that the lexicographically ordered monomials of the form
\begin{align*}
\prod_{p_j<0 }\widehat{h}_{i_j,(p_j)} \prod_{q_j<0 }e^R_{\alpha_j,(q_j)}\prod_{r_j\leq 0 }a^*_{\alpha_j,r_j} \mathbf{1} 
\end{align*}
form a basis of $\sA\otimes \pi^{0}_\h$ where we set $\widehat{h}_{i_j}=\w{\rho}(h_{i_j})$ runs the basis of $\h$. 
Since 
$$\cent(V^{\kk_c}(\g))\subset \Com(V^0(\nil_+), V^{\kk_c}(\g)),$$ Proposition \ref{Wakimoto for affine I} implies that $\cent(V^{\kk_c}(\g))$ embeds into the subspace spanned by 
\begin{align*}
\prod_{p_j<0 }\widehat{h}_{i_j,p_j} \prod_{q_j<0 }e^R_{\alpha_j,q_j} \mathbf{1}.
\end{align*}
It follows from $\cent(V^{\kk_c}(\g))\subset \Com(V^{\kk_c}(\h), V^{\kk_c}(\g))$ that $\cent(V^{\kk_c}(\g))$ lies in the subspace of $\h$-weight 0. 
As $e^R_{\alpha}$ has $\h$-weight $\alpha$, it follows that $\cent(V^{\kk_c}(\g))$ lies in the subspace spanned by $\prod_{p_j<0 }\widehat{h}_{i_j,p_j} \mathbf{1}$, that is, $\pi^0_\h$. This completes the proof.
\endproof

\subsection{\tmath{$\W$}-superalgebra case}

Next, we apply the quantum Hamiltonian reduction $\HH_{\OO}$ for the regular nilpotent element $f$ to the Wakimoto realization in Proposition \ref{Wakimoto for affine I} and obtain a homomorphism 
\begin{align*}
[\w{\rho}]\colon \W^\kk(\g)\rightarrow \HH_{\OO}^0(\sA\otimes \pi^{\kk-\kk_c}_\h).
\end{align*}
Let us now change the coordinates of $\C[N_+]$ so that
\begin{align*}
    \n_+= \n_+^0 \rtimes \n_+^+ \overset{\sim}{\longrightarrow}N_+^0 \rtimes N_+ \simeq N_+
\end{align*}
by using the good grading and, accordingly, decompose $\sA\simeq \sA_\OO\otimes \sA_+$
where 
$$\sA_\OO=\langle a_\alpha, a_\alpha^* \mid \alpha \in \Delta_+^0\rangle,\quad \sA^+=\langle a_\alpha, a_\alpha^* \mid \alpha \in \Delta_+^+\rangle.$$
\begin{lemma}
    There is an isomorphism of vertex superalgebras
\begin{align*}
\HH_{\OO}^n(\sA\otimes \pi^{\kk-\kk_c}_\h)\simeq 
\delta_{n,0} (\sA_\OO\otimes \pi^{\kk-\kk_c}_\h).
\end{align*}
Moreover,
 \begin{align*}
\DS^n(\affWak{\lambda})\simeq 
\delta_{n,0} \affWak{\OO,\lambda},\qquad \affWak{\OO,\lambda}:=\sA_\OO\otimes \pi^{\kk-\kk_c}_{\h,\lambda}
\end{align*}
holds as modules over $\affWak{\OO}:=\sA_\OO\otimes \pi^{\kk-\kk_c}_\h$.
\end{lemma}

\begin{proposition}\label{Wakimoto for principal super I}
The homomorphism 
$$[\w{\rho}]\colon \W^\kk(\g)\rightarrow \affWak{\OO}=\sA_\OO\otimes \pi^{\kk-\kk_c}_\h$$
is an embedding. Moreover, the image is characterized by the kernel of the screening operators at generic levels:
\begin{align*}
\W^\kk(\g) \simeq \bigcap \Ker S_i^\OO \subset \sA_\OO \otimes \heis{\kk-\kk_c}
\end{align*}
$$S_i^\OO=\int Y([e_{i}^R]\mathrm{e}^{-\frac{1}{k+n}\alpha_i},z)\mathrm{d}z\colon \affWak{\OO,0}\rightarrow \affWak{\OO,-\alpha_i},\quad (i=1,\cdots,n).$$
with 
\begin{align*}
[e_i^R]=\mathbf{1}, \quad (i=1,\dots, n-1),\qquad [e_n^R]=a_{\alpha_{n}}.
\end{align*}
\end{proposition}
\proof
The proof is obtained straightforwardly by using Proposition \ref{Wakimoto for affine I} and the free field realization obtained in \cite{CGN21}.
\endproof
Note that $\sA_\OO= \langle a_{\alpha_n}, a_{\alpha_n}^* \rangle$ and thus 
\begin{align*}
    \sA_\OO \simeq V_\Z,\quad a_{\alpha_n}, a_{\alpha_n}^*\mapsto \hwt{x}, \hwt{-x}
\end{align*}
by the boson-fermion correspondence. In the rest of this paper, we identify these two realizations. 
Thus, this realization agrees with the one in Theorem \ref{thm: Wakimoto for sW}, and we will identify them in the below.

Since the center
\begin{align*}
\cent(V^{\kk_c}(\g))\subset V^{\kk_c}(\g) \subset C_{f}^\bullet(V^{\kk_c}(\g))
\end{align*}
lies in the center of the BRST complex, we have, in particular, $\cent(V^{k_c}(\g))\subset \Ker\ d$.
Thus, we have the natural homomorphism
\begin{align*}
\iota\colon \cent(V^{\kk_c}(\g)) \rightarrow \cent(\W^{\kk_c}(\g))\subset \W^{\kk_c}(\g). 
\end{align*}
\begin{lemma}\label{FF center lies in W-center}
The map $\iota$ is injective.
\end{lemma}
\proof
The assertion follows from the Wakimoto realizations for both hand-sides in Lemma \ref{FF center goes Heis} and Theorem \ref{thm: Wakimoto for sCoset}, thanks to the following commutative diagram:
    \begin{center}
		\begin{tikzcd}
			\cent(V^{\kk_c}(\g))
			\arrow[d, hook, "\w{\rho}"']
			\arrow[r, "\iota"]&
			\cent(\W^{\kk_c}(\g)) \arrow[d, hook, "{[\w{\rho}]}"]
                \arrow[r,symbol=\subset]&
			\W^{\kk_c}(\g) \arrow[d, hook, "{[\w{\rho}]}"]
			\\
		      \pi^0_{\h} \arrow[r, "\mathrm{id}"]&
			  \pi^0_{\h} \arrow[r,symbol=\subset]&
			\sA_\OO\otimes \pi^0_{\h}.
		\end{tikzcd}
	\end{center}
\endproof

\section{Center of affine \tmath{$\gl_{n|1}$} at the critical level}\label{sec: Center of affine at the critical level}
\subsection{Segal--Sugawara vectors}
In \cite{MR}, Molev and Ragoucy constructed a family of elements in $\cent(V^{\kk_c}(\gl_{m|n}))$ by exploiting the following matrix-valued differential operator
\begin{align*}
    \pd+ \widehat{E}=\sum_{i,j}e_{i,j}\otimes (\delta_{i,j}\pd+ e_{i,j(-1)})(-1)^{p(i)}
\end{align*}
in $R:=\End(\C^{m|n})\otimes U(\widehat{\gl}_{m|n}\rtimes \C\pd)$. Here, the second factor is the enveloping algebra of the Lie superalgebra $\widehat{\gl}_{m|n}\rtimes \C\pd$ and the product on  $R$ is defined through the Koszul sign rule:
$$(a\otimes b)(a'\otimes b')=(-1)^{p(b) p(a')}aa'\otimes bb'$$
    for parity homogeneous elements $a,a',b,b'$.
We consider the supertrace on the first factor:
\begin{align*}
    \str\colon R\rightarrow U(\widehat{\gl}_{m|n}\rtimes \C\pd)
\end{align*}
and compose it with the projection
 \begin{align*}
     U(\widehat{\gl}_{m|n}\rtimes \C\pd)\simeq U(\widehat{\gl}_{m|n})\otimes \C[\pd]\twoheadrightarrow V^\kk(\gl_{m|n})\otimes \C[\pd]
 \end{align*}
 to obtain the supertrace
 \begin{align*}
    \str\colon R\rightarrow U(\widehat{\gl}_{m|n}\rtimes \C\pd)\twoheadrightarrow V^\kk(\gl_{m|n})\otimes \C[\pd]
\end{align*}
by abuse of notation.

\begin{theorem}[\cite{MR}]\label{Segal-Sugawara vectors}
At the critical level $\kk=\kk_c$, the coefficients $s_{p,q}$ in the expansions 
\begin{align*}
    \str(\pd+\widehat{E})^p=s_{p,0}\pd^n+ s_{p,1}\pd^{p-1}+\dots+s_{p,p}
\end{align*}
for $p\geq0$ lie in the center $\cent(V^{\kk_c}(\gl_{m|n}))$.
\end{theorem}

\subsection{Supersymmetric polynomials}\label{sec: supersymmetric polynomials}
The center $\cent(V^{\kk_c}(\gl_{m|n}))$ is closely related to the affine supersymmetric polynomials as observed in \cite{MM} for $\widehat{\gl}_{1|1}$.
Introduce the polynomial ring $R^{m|n}$ by
$$R^{m|n}:=\C[u_1,\dots,u_m,v_1,\dots, v_n],$$
which admits the natural action of the symmetric group $\mathfrak{S}_{m+n}$.
The ring of supersymmetric polynomilas $\Lambda^{m|n}$ is the subalgebra 
\begin{align*}
\Lambda^{m|n}:=\left\{f ; f|_{u_m=t,v_n=-t}\in R^{m-1|n-1} \right\}\subset (R^{m|n})^{\mathfrak{S}_m\times \mathfrak{S}_n} \subset R^{m|n}.
\end{align*}
It is known \cite{Ma} that $\Lambda^{m|n}$ is generated by the power sums 
\begin{align*}
    s_p=(u_1^p+\cdots +u_m^p)-(-1)^p(v_1^p+\cdots +v_n^p)\quad (p>0).
\end{align*}
The ring of affine supersymmetric polynomials $\Lambda^{m|n}_{\aff}$ is the (commutative) vertex algebraic extension of $\Lambda^{m|n}$; we consider the commutative vertex algebra
$$R^{m|n}_{\infty}:=\C[\pd^N u_1,\dots,\pd^N u_m, \pd^Nv_1,\dots, \pd^Nv_n\mid N\geq0]\simeq \C[J_\infty \h]$$
and introduce $\Lambda^{m|n}_{\aff}$ as the subalgebra generated by $\Lambda^{m|n}$ as a differential algebra:
$$\Lambda^{m|n}_{\aff}:=\langle \pd^N f\mid f\in \Lambda^{m|n}, N\geq 0 \rangle \subset R^{m|n}_{\infty}.$$

Now, we relate the affine supersymmetric polynomials to the Segal--Sugawara vectors Theorem \ref{Segal-Sugawara vectors} through the Wakimoto realization in Proposition \ref{Wakimoto for affine I}. 
We take the Poisson vertex superalgebras on the associated graded (super)vector spaces with respect to Li filtration (see \ref{sec: general coset}) in Proposition \ref{Wakimoto for affine I}, and consider the induced homomorphism
\begin{align*}
    \overline{\Psi}\colon \gr V^\kk(\gl_{n|1})\rightarrow \gr (\mathcal{A}\otimes \pi^{k+\dC}_\h)\simeq  \gr \mathcal{A}\otimes \gr \pi^{k+\dC}_\h.
\end{align*}

\begin{proposition}\label{affine ss polys in FF center}
There is an embedding
$$\Lambda^{n|1}_{\aff}\subset \gr\cent(V^{\kk_c}(\gl_{n|1})).$$
\end{proposition}
\proof
Note that one has an isomorphism $R_\infty^{n|1} \xrightarrow{\simeq} \gr \pi^{0}_\h$ which maps
\begin{align*}
     u_1,\dots, u_n \mapsto e_{1,1},\dots,e_{n,n},\quad v_1 \mapsto  e_{n+1,n+1}.
\end{align*}
Hence, it suffices to show that $\Im \overline{\Psi}$ contains the power sums $s_{p}$ ($p\geq0$) as $\overline{\Psi}$ is an embedding and restricts to 
\begin{align*}
    \overline{\Psi}\colon \gr\cent(V^{\kk_c}(\gl_{n|1})) \hookrightarrow \gr \pi^{0}_\h
\end{align*}
by Lemma \ref{FF center goes Heis}. 
By the formula in Proposition \ref{Wakimoto for affine I},
the image of the Segal--Sugawara vector $s_{p,p}$ agrees with 
\begin{align*}
    \str \left(\begin{array}{cccc}
        e_{1,1} &  & & \\
         & \ddots & &\\
         & & e_{n,n} &\\
         &&& -e_{n+1.n+1}
    \end{array} \right)^p 
    &= e_{1,1}^p+\cdots +e_{n,n}^p-(-e_{n+1,n+1})^p\\
    &= u_1^p+ +\cdots +u_n^p-(-1)^pv_1^p=s_p.
\end{align*}
This completes the proof.
\endproof

\subsection{Center}\label{sec: Center}
Now, we are ready to establish the main results of the paper.
We start with the center at the level of the associated graded vector space for Lie filtration.

\begin{theorem}\label{them: Detecting the FF center}\hspace{0mm}\\
\textup{(1)} The homomorphism $\iota$ is an isomorphism
\begin{align*}
\iota\colon \cent(V^{\kk_c}(\gl_{n|1}))\xrightarrow{\simeq} \cent(\W^{\kk_c}(\gl_{n|1})).
\end{align*}
\textup{(2)} There is an isomorphism of differential algebras 
\begin{align*}
 \gr\cent(V^{\kk_c}(\gl_{n|1}))\simeq \Lambda^{n|1}_{\aff}.
\end{align*}
\end{theorem}
\proof
By Theorem \ref{thm: Wakimoto for sCoset} and Lemma \ref{FF center lies in W-center}, we have embeddings
\begin{align*}
    \cent(V^\kk(\gl_{n|1}))\subset \cent(\W^\kk(\gl_{n|1})) \subset \pi^0_\h.
\end{align*}
Since $\Lambda^{n|1}_{\aff}\subset \gr\cent(V^{\kk_c}(\gl_{n|1}))$ by Proposition \ref{affine ss polys in FF center}, it suffices to show the opposite inclusion
\begin{align}\label{opposite inclusion}
    \gr\cent(\W^{\kk_c}(\gl_{n|1}))\subset \Lambda^{n|1}_{\aff}.
\end{align}
By Theorem \ref{critical level structure of reg superW}, we have the isomorphism 
\begin{align}\label{isom in the proof}
    \gr \W^{\kk_c}(\gl_{n|1})\simeq \gr \C[J_\infty \mathbb{A}^{n-1}] \otimes \gr V^{\kk_c}(\gl_{1|1})
\end{align}
of Poisson vertex superalgebras.
By \cite{MM}, there is an isomorphism $\gr \cent(V^{\kk_c}(\gl_{1|1}))\simeq \Lambda^{1|1}_\aff$
and $\gr \cent(V^{\kk_c}(\gl_{1|1}))$ is generated by the elements
$$\overline{h}_{p,p}=e_{1,1}^{p-1}(e_{1,1}+e_{2,2})+(p-1)e_{1,1}^{p-2}e_{2,1}e_{1,2}$$
for $p\geq 1$. We apply the formula under the embedding $V^{\kk_c}(\gl_{1|1})\hookrightarrow \W^{\kk_c}(\gl_{n|1})$ in  \eqref{eq: subalg 2}, we obtain 
\begin{align*}
  \overline{h}_{p,p}
  &=J^{p-1}\overline{\mathscr{W}}_n+(p-1) \overline{J}^{p-2}\overline{G}_\tn \overline{G}_\tp\\
  &=(-\overline{u}_{n+1}+\overline{\psi} \overline{\psi}^*)^{p-1} \overline{\mathscr{W}}_n+(p-1)(-\overline{u}_{n+1}+\overline{\psi} \overline{\psi}^*)^{p-2}\overline{\mathscr{W}}_n \overline{\psi}^* \overline{\psi}\\
  &=(-1)^{p-1}\overline{u}_{n+1}^{p-1} \overline{\mathscr{W}}_n
\end{align*}
in $\gr \mathcal{A} \otimes \gr \pi^0_\h$.
As $\mathscr{W}_n=(-1)^n (\pd+H_1)\cdots (\pd+H_n)$, we have $\overline{\mathscr{W}}_n=\overline{H}_1\cdots \overline{H}_n$ and thus $\overline{h}_{p,p}$ is identified as 
$$\overline{h}_{p,p}=(-1)^{p-1}v_{1}^{p-1}(u_1+v_1)\cdots (u_n+v_1)$$
in $R^{n|1}_\infty$. 
Since $\overline{h}_{p,p}|_{u_n=t,v_1=-t}=0$, $\overline{h}_{p,p}$ lies in $\Lambda^{n|1}_\aff$.
On the other hand, the generators $\overline{W}_p$ ($p=1,\cdots,n-1$) of the first factor in \eqref{isom in the proof} are realized as coefficients
\begin{align*}
  (z-u_1)\cdots (z- u_n) (z+u_{n+1})^{-1}=z^{n-1}+\sum_{n=1}^\infty \overline{W}_p \partial^{n-1-p}
\end{align*}
with $z$ an indeterminate by \eqref{eq: gln1 ps diffop}. By substituting $u_n=t$ and $u_{n+1}=-t$, the right-hand side gets $ (z-u_1)\cdots (z- u_{n-1})$ and thus independent of $t$. Thus, so are the coefficients $\overline{W}_p$'s, i.e. they are in $\Lambda^{n|1}_{\aff}$ by definition. 
Hence, we obtain \eqref{opposite inclusion}. 
\endproof

As a consequence, we obtain the following main result.
\begin{theorem}\label{THM: FF center for gln1 vis pD op}
The center $\cent(V^{\kk_c}(\gl_{n|1}))$ is isomorphic to the subalgebra in $\pi^0_\h$
strongly generated by the coefficients $W_p$ $(p\geq1)$ of the the pseudo-differential operator $\mathcal{L}$:
\begin{align}\label{eq: pD op}
\mathcal{L}=(\partial- u_1)\cdots (\partial- u_n) (\partial+u_{n+1})^{-1} 
 =\partial^{n-1}+ \sum_{p=1}^\infty W_p \partial^{n-1-p}.
\end{align}
\end{theorem}
\proof
By Theorem \ref{them: Detecting the FF center}, we have $\cent(V^{\kk_c}(\gl_{n|1}))\simeq \cent(\W^{\kk_c}(\gl_{n|1}))$ and that $\cent(\W^{\kk_c}(\gl_{n|1}))$ contains a subalgebra generated by $W_p$ ($p\geq1$) by Proposition \ref{prop: coset for Wsalg}. 
By taking the associated graded vector spaces $\gr \cent(V^{\kk_c}(\gl_{n|1}))\simeq \gr \cent(\W^{\kk_c}(\gl_{n|1}))$, we have shown in the proof of Theorem \ref{them: Detecting the FF center} that the images $\overline{W}_p$'s generate the affine supersymmertic polynomials $\Lambda^{n|1}_{\aff}$, and thus $\gr \cent(V^{\kk_c}(\gl_{n|1}))$ by Theorem \ref{them: Detecting the FF center} (2). 
Hence, $W_p$'s strongly generate the center $\cent(V^{\kk_c}(\gl_{n|1}))$. 
\endproof

Let us denote by $\mathrm{Fun}(\mathrm{Op}_{n|1}(D))$ the differential subalgebra of $\pi^0_\h$ generated by $W_p$ ($p\geq 1$), whose spectrum $\mathrm{Op}_{n|1}(D)$ paramtrize the pseudo-differential operators of the form \eqref{eq: pD op} over the local disc $D=\mathrm{Spec}\ \C(\!(z)\!)$.

Since the embedding in Theorem \ref{THM: FF center for gln1 vis pD op} is an analogue of the Harish-Chandra map 
$\iota_{\mathrm{HC}}\colon \cent(\gl_{n|1})\xrightarrow{\simeq} \Lambda^{n|1}\subset U(\h)$
in the finite setting (see e.g.\ \cite{Musson}), which one may recover by taking Zhu's functor, we call the embedding in Theorem \ref{THM: FF center for gln1 vis pD op} the affine Harish-Chandra map: 
\begin{align}\label{eq: affine HC map}
    \widehat{\iota}_{\mathrm{HC}}\colon \cent(V^{\kk_c}(\gl_{n|1})) \xrightarrow{\simeq} \mathrm{Fun}\left(\mathrm{Op}_{n|1}(D) \right)\subset \pi^0_\h.
\end{align}

\subsection{Poisson structure on the center}\label{sec: Poisson structure on the center}
Although the center $\cent(V^{\kk_c}(\gl_{n|1}))$ at the critical level in Theorem \ref{THM: FF center for gln1 vis pD op} is just a commutative vertex algebra (i.e.,\ differential algebra), it naturally inherits a Poisson vertex algebra structure through the affine Harish-Chandra map $\widehat{\iota}_{\mathrm{HC}}$. 
Note that $\pi^0_\h$ is strongly generated by $u_1,\cdots, u_{n+1}$ and inherited with a standard Poisson vertex algebra by the formula 
\begin{align}\label{eq: std Poisson structure}
\plm{u_i}{u_j}=\kappa_V(u_i,u_j)\lambda,
\end{align}
see e.g. \cite{FBZ04, Kac15} for details of Poisson vertex algebras.
\begin{proposition}\label{prop: Poisson str on the center}
The differential subalgebra $\mathrm{Fun}(\mathrm{Op}_{n|1}(D))$ is a Poisson vertex subalgebra of $\pi^0_\h$. This inherites a Poisson vertex algebra structure on $\cent(V^{\kk_c}(\gl_{n|1}))$ by the affine Harish-Chandra map.
\end{proposition}
Note that the affine Harish-Chandra map \eqref{eq: affine HC map} is the specialization of the description of the Heisenberg coset $C^\kk(\gl_{n|1})$ in Theorem \ref{prop: coset for Wsalg}.
Recall that $C^\kk(\gl_{n|1})$ is strongly generated at generic levels by the coefficients of the expansion of the pseudo-differential operator
\begin{align}\label{eq: PDO for deformed center}
    \mathcal{L}_k=(\partial+\varepsilon t_1)\cdots (\partial+\varepsilon t_n) (\partial+\varepsilon t_{n+1})^{\varepsilon}
\end{align}
with $\varepsilon=\frac{1}{-1+(k+\dC)}$, which we denote by $\mathrm{Fun}(\mathrm{Op}^\kk_{n|1}(D))$.
By definition of $t_i$'s in \eqref{eq: OPE for ti's}, we have the following diagram:
    \begin{center}
		\begin{tikzcd}
			C^\kk(\gl_{n|1})
			\arrow[r, "\simeq"] \arrow[d,dashed , "{\kk \rightarrow \kk_c}"]&
			\mathrm{Fun}(\mathrm{Op}^\kk_{n|1}(D)) \arrow[d,dashed , "{\kk \rightarrow \kk_c}"]
                \arrow[r,symbol=\subset]&
			\pi^{k+\dC}_{\tori} \arrow[d,dashed , "{\kk \rightarrow \kk_c}"]
			\\
		      \cent(V^{\kk_c}(\gl_{n|1})) \arrow[r, "\simeq"]&
			  \mathrm{Fun}(\mathrm{Op}_{n|1}(D)) \arrow[r,symbol=\subset]&
			 \pi^0_{\h},
		\end{tikzcd}
	\end{center}
which summarizes Theorem \ref{prop: coset for Wsalg}, \ref{them: Detecting the FF center} and \ref{THM: FF center for gln1 vis pD op}.

\proof[Proof of Proposition \ref{prop: Poisson str on the center}]
Note that the vertex algebra $\pi^{k+\dC}_{\tori}$ defined over $\C$ depending on $k$ has an integral form, namely, the vertex algebra $\pi^\mathbf{O}_{\tori}$ defined and free over the ring $\mathbf{O}=\C[\mathbf{x}]$ and generated by $\mathbf{t}_i$ ($i=1,\cdots,n+1$) satisfying the $\lambda$-brackets \eqref{eq: OPE for ti's} under the modification $t_i \mapsto \mathbf{t}_i$ and $(k+\dC)\mapsto \mathbf{x}$.
By using $\C_{k}=\mathbf{O}/(\mathbf{x}-(k+\dC))$ ($k\in \C$), we recover $\pi^{k+\dC}_{\tori}\simeq \pi^{\mathbf{O}}_{\tori}\otimes_{\mathbf{O}}\C_k$ as vertex algebras. In particular, the case of the critical level is recovered as $\pi^{0}_{\h}\simeq \pi^{\mathbf{O}}_{\tori}\otimes_{\mathbf{O}}\C_0$, which holds after the completion $\mathbf{O}^\wedge=\C[\![\mathbf{x}]\!]\supset \mathbf{O}$, i.e., $\pi^{0}_{\h}\simeq \pi^{\mathbf{O}^\wedge}_{\tori}\otimes_{\widehat{\mathbf{O}}}\C_0$. 
Since $\pi^{0}_{\h}$ is commutative, we have 
\begin{align*}
    \lm{\cdot}{\cdot}\colon \pi^{\mathbf{O}^\wedge}_{\tori}\times \pi^{\mathbf{O}^\wedge}_{\tori} \rightarrow x\cdot \pi^{\mathbf{O}^\wedge}_{\tori}
\end{align*}
Thus the composition 
\begin{align*}
    \pi^{\mathbf{O}^\wedge}_{\tori}\times \pi^{\mathbf{O}^\wedge}_{\tori}\overset{\lm{\cdot}{\cdot}}{\rightarrow }  \mathbf{x}\cdot \pi^{\mathbf{O}^\wedge}_{\tori}[\lambda] \overset{1/\mathbf{x}}{\rightarrow } \pi^{\mathbf{O}^\wedge}_{\tori}[\lambda] \overset{\cdot \otimes 1}{\rightarrow }(\pi^{\mathbf{O}^\wedge}_{\tori}\otimes_{\mathbf{O}^\wedge}\C_0)[\lambda]\simeq \pi^{0}_{\h}[\lambda]
\end{align*}
is well-defined and induces a bilinear map
\begin{align*}
   \plm{\cdot}{\cdot}\colon \pi^{0}_{\h} \times \pi^{0}_{\h} \rightarrow \pi^{0}_{\h}[\lambda],
\end{align*}
which gives a Poisson $\lambda$-bracket on $\pi^{0}_{\h}$ and indeed agrees with the one in \eqref{eq: std Poisson structure}. 
Together with the differential algebra structure on $\pi^{0}_{\h}$ similarly obtained, the specialization $\pi^{0}_{\h}\simeq \pi^{\mathbf{O}}_{\tori}\otimes_{\mathbf{O}}\C_0$ gives the standard Poisson vertex algebra structure on $\pi^0_\h$, see \cite{FBZ04} for details. 
Now, the Heisenberg coset subalgebra $C^\kk(\gl_{n|1})\subset \pi^{k+\dC}_\h$ is a family of vertex subalgebras depending on the levels $k+\dC\neq 1$ which admits the set of strong generators $W_p$ ($p\geq0$) at generic levels defined uniformly for all $k+\dC\neq 1$.
Hence, we may find at a localization $\mathbf{O}_{loc}=\C[\mathbf{x}][\frac{1}{\mathbf{x}-a}\mid a \in R]$ for some subset $R\subset \C$ and a vertex subalgebra 
$\mathrm{Fun}(\mathrm{Op}^{\mathbf{O}_{loc}}_{n|1}(D))\subset \pi^{\mathbf{O}}_\tori\otimes_\mathbf{O}\mathbf{O}_{loc}  $
generated by $W_p$ obtained as coefficients of \eqref{eq: PDO for deformed center} under the replacement $t_i \mapsto \mathbf{t}_i$ and $\varepsilon\mapsto 1/(-1+\mathbf{x})$.
Note that the critical level is avoided, i.e., $0 \notin R$, thanks to Theorem \ref{THM: FF center for gln1 vis pD op}. 
Hence, we may take the completion $\mathrm{Fun}(\mathrm{Op}^{\mathbf{O}^\wedge}_{n|1}(D))$ of $\mathrm{Fun}(\mathrm{Op}^{\mathbf{O}_{loc}}_{n|1}(D))$ and 
\begin{align*}
    \mathrm{Fun}(\mathrm{Op}_{n|1}(D)) \simeq \mathrm{Fun}(\mathrm{Op}^{\mathbf{O}^\wedge}_{n|1}(D))\otimes_{\widehat{\mathbf{O}}}\C_0.
\end{align*}
Thus, the vertex algebra strucutre on $\mathrm{Fun}(\mathrm{Op}^{\mathbf{O}^\wedge}_{n|1}(D))$ induces a Poisson vertex algebra structure on $\mathrm{Fun}(\mathrm{Op}_{n|1}(D))$ as a Poisson vertex subalgebra of the standard one on $\pi^0_\h$. This completes the proof.
\endproof

\subsection{Digression on the large level limit}
Here, we describe the Heisenberg coset $C^{\kk_c}(\sll_{n|1})$ (resp.\ $C^{\kk_c}(\gl_{n|1})$) in terms of the large level limit of the $\W$-algebra $\W^{\infty}(\sll_{n},\OO_\sub)$ (resp. $\W^{\infty}(\gl_{n},\OO_\sub)$) identified with the regular function of the jet scheme of the Slodowy slice $\mathscr{S}_{\sub}$ associated with the nilpotent element $f$ in \eqref{eq: subregular nilp}. 

To formulate the large level limit of $\W^{\infty}(\gl_{n},\OO_\sub)$, we use the Wakimoto realization $V^{\check{\kk}}(\gl_n)$ which admits an integral form, allowing the (Poisson vertex algebra) limit $\check{\kk}\rightarrow \infty$, following \cite{FF2, F}. The Wakimoto realization for $V^{\check{\kk}}(\gl_n)$ is an embedding of vertex algebras
\begin{align}\label{eq: affine Wakimoto for non-super}
\Psi_{\mathrm{aff}}\colon V^{\check{\kk}}(\gl_n) \hookrightarrow \sA_{\Delta^+}\otimes \pi^{\check{\kk}+\check{\dC}}_{\check{\h}},
\end{align}
where $\sA_{\Delta^+}$ is the tensor product of the $\beta\gamma$-systems  $\sA$ indexed by the positive roots $\Delta_+$.
Let us introduce the ring $\dso=\C[\frac{1}{\dx}]$, where $\dx$ is a formal variable corresponding to $\check{\kk}+\dC$ after the evaluation $\dso \rightarrow \C_{\check{\kk}}$,  $\frac{1}{\dx}\mapsto \frac{1}{\check{\kk}+\check{\dC}}$. Let us introduce the following vertex algebras over $\dso$
\begin{align}\label{eq: some vertex algebras over ring}
    V^\dso(\gl_n),\quad \sA_{\Delta^+}^\dso,\quad \pi^{\dso}_{\check{\h}},
\end{align}
which are free $\dso$-algebras strongly generated by 
\begin{align*}
    \w{u}=\frac{1}{\dx}u,\quad, \w{a}_\alpha= \frac{1}{\dx}a_\alpha,\ \w{a}_\alpha^*=a_\alpha^*,\quad \w{h}_i=\frac{1}{\dx}h_i
\end{align*}
for $u \in \sll_n$, $\alpha \in \Delta_+$, and $1\leq i \leq n-1$. By definition, they satisfy the $\lambda$-brackets
\begin{align*}
    &\lm{\w{u}}{\w{v}}=\frac{1}{\dx}\left(\widehat{[u,v]}+ \frac{\dx-\dC}{\dx}\kappa_V(u,v)\lambda\right),\\
    &\lm{\w{a}_\alpha}{\w{a}_\beta^*}=\frac{1}{\dx}\delta_{\alpha,\beta},\quad
    \lm{\w{h}_i}{\w{h}_j}=\frac{1}{\dx}\kappa_V(h_i,h_j)\lambda.
\end{align*}
Hence, we may take the evaluation $\check{\kk}=\infty$ as $\frac{1}{\dx}=0$, which we denote by $\dso \rightarrow \C_{\infty}$. 
Likewise, let us denote by $V^\infty(\gl_n), \sA_{\Delta^+}^\infty, \pi^{\infty}_{\check{\h}}$ the evaluations of those in \eqref{eq: some vertex algebras over ring}. Note that these differential algebras in the limits have a natural Poisson vertex algebra by a similar formula in Section \ref{sec: Poisson structure on the center}.
Then, the embedding in \eqref{eq: affine Wakimoto for non-super} is lifted to an embedding
\begin{align*}
    \Psi_{\mathrm{aff}}^{\dso}\colon V^\dso(\gl_n)\rightarrow \sA_{\Delta^+}^\dso \otimes_{\dso}\pi^{\dso}_{\check{\h}}
\end{align*}
by \cite{FF2} which recovers \eqref{eq: affine Wakimoto for non-super} by evaluations $-\otimes_{\dso}\C_{\check{\kk}}$. 
By applying the BRST reduction $\HH_{\OO_{\mathrm{sr}}}^0$, we obtain an embedding of vertex algebras 
\begin{align}\label{eq: Wakimoto intergal form}
    [\Psi_{\mathrm{aff}}^{\dso}]\colon \W^\dso(\gl_n,\OO_{\mathrm{sr}})\rightarrow \sA^\dso \otimes_{\dso}\pi^{\dso}_{\check{\h}},
\end{align}
whose evaluations $-\otimes_{\dso}\C_{\check{\kk}}$ recovers the Wakimoto realization of the $\W$-algebra $\W^{\check{\kk}}(\gl_n,\OO_{\mathrm{sr}})$ in Proposition \ref{Wakimoto for subreg 2}, see \cite{G20}. 
Let us introduce the large level limit 
$$\W^{\infty}(\gl_n,\OO_{\mathrm{sr}})=\W^{\check{\kk}}(\gl_n,\OO_{\mathrm{sr}})\otimes_{\dso}\C_{\infty}.$$
Since $\W^{\check{\kk}}(\gl_n,\OO_{\mathrm{sr}})$ contains the Heisenberg field $\check{\mathbf{H}}=\frac{1}{\dx}(aa^*+\check{\varpi}_{n-1})$ (see \S \ref{sec: Subregular Walg}), which satisfies the $\lambda$-bracket $\lm{\check{\mathbf{H}}}{\check{\mathbf{H}}}=\frac{1}{\dx^2}(1-\frac{n-1}{n}\dx)\lambda$, we have the complete reducibility of weight modules over the (nondegenerate) Heisenberg vertex algebras
\begin{align}\label{eq: integral form and branching rule}
    \W^{\dso_{loc}}(\gl_n,\OO_{\mathrm{sr}}) \simeq \bigoplus_{n \in \Z} C^{\dso_{loc}}_n(\gl_n,\OO_{\mathrm{sr}})\otimes_{\dso_{loc}} \pi^{\dso_{loc}}_{\check{\mathbf{H}},n}
\end{align}
after the localization $-\otimes_\dso \dso_{loc} $ where $\dso_{loc}=\dso[(\frac{1}{\dx}-\frac{n-1}{n})^{-1}]$. Here, $\pi^{\dso_{loc}}_{\check{\mathbf{H}},n}$ is the Fock module over $\pi^{\dso_{loc}}_{\check{\mathbf{H}}}$ of highest weight $n/\dx$, and $C^{\dso_{loc}}_n(\gl_n,\OO_{\mathrm{sr}})$ is the free $\dso_{loc}$-submodule of $\W^{\dso_{loc}}(\gl_n,\OO_{\mathrm{sr}})$ consisting of highest weight vector for $\pi^{\dso_{loc}}_{\check{\mathbf{H}}}$ of weight $n/\dx$. 
Note that $C^{\dso_{loc}}_0(\gl_n,\OO_{\mathrm{sr}})$ is a vertex subalgebra and satisfies 
\begin{align*}
    C^{\dso_{loc}}_0(\gl_n,\OO_{\mathrm{sr}})\otimes_{\dso_{loc}}\C_{\check{\kk}}&\simeq C^{\check{\kk}}(\gl_n,\OO_{\mathrm{sr}})
\end{align*}
as vertex algebras, where the right-hand side denotes the Heisenberg coset subalgebra
$$C^{\check{\kk}}(\gl_n,\OO_{\mathrm{sr}})=\Com (\pi_{\check{H}},\W^{\check{\kk}}(\gl_n,\OO_{\mathrm{sr}})).$$
Let us introduce the large level limit as 
\begin{align*}
    C^{\infty}(\gl_n,\OO_{\mathrm{sr}})=C^{\dso_{loc}}_0(\gl_n,\OO_{\mathrm{sr}})\otimes_{\dso_{loc}}\C_{\infty},
\end{align*}
which is regarded as a multiplicity space since 
\begin{align}\label{eq: branching rule at infinity}
    \W^{\infty}(\gl_n,\OO_{\mathrm{sr}}) \simeq \bigoplus_{n \in \Z} C^{\infty}_n(\gl_n,\OO_{\mathrm{sr}})\otimes_{\C} \pi^{\infty}_{\check{\mathbf{H}},n}
\end{align}
by applying $-\otimes_{\dso_{loc}}\C_{\check{\kk}}$ to \eqref{eq: integral form and branching rule}. 
\begin{remark}
\textup{
We have natural isomorphisms of Poisson vertex algebras
$$\W^{\infty}(\sll_n,\OO_{\mathrm{sr}})\simeq \HH_{\OO_{\mathrm{sr}}}^0(V^\infty(\sll_n))\simeq \C[J_\infty \mathscr{S}_{\sub}],$$
where  $J_\infty \mathscr{S}_{\sub}$ is the jet scheme of the Slodowy slice $\mathscr{S}_{\sub}$ associated with the subregular nilpotent element $f$, see \cite{Ar15}.
The slice $\mathscr{S}_{\sub}$ admits a moment map $\mu \colon  \mathscr{S}_{\sub} \rightarrow \gl_1^*$ which induces a morphism on the jet scheme 
$\mu_\infty\colon J_\infty\mathscr{S}_{\sub}\rightarrow J_\infty \gl_1^*$.
This corresponds to the embedding $\pi^{\infty}_{\check{\mathbf{H}}}\hookrightarrow \W^{\infty}(\sll_n,\OO_{\mathrm{sr}})$. Hence, \eqref{eq: branching rule at infinity} implies an isomorphism 
\begin{align*}
    C^{\infty}(\gl_n,\OO_{\mathrm{sr}})\simeq \C[\mu_{\infty}^{-1}(0)/\!/\C^\times].
\end{align*}
}
\end{remark}

The following was conjectured by two of the authors in \cite{AN-2024}.
\begin{proposition} For $n>1$, there is an isomorphism
    \begin{align*}
        C^{\infty}(\gl_{n},\OO_{\mathrm{sr}}) \simeq C^{\kk_c}(\gl_{n|1})\ (\simeq \cent(V^{\kk_c}(\gl_{n|1}))).
    \end{align*}
\end{proposition}
\proof
Since both sides admit the decompositions induced by $\gl_{n|m}=\sll_{n|m}\oplus \C \Omega$, it suffices to show $C^{\infty}(\sll_{n},\OO_{\mathrm{sr}}) \simeq C^{\kk_c}(\sll_{n|1})$. By applying the Heisenberg coset $\Com(\pi_J,-)$ in Theorem \ref{FS duality}, we obtain isomorphisms of vertex algebras
    \begin{center}
\begin{tikzcd}[ column sep = huge]
			\pi_{\widetilde{\alpha}_1,\cdots,\widetilde{\alpha}_n}
			\arrow[r, " \overset{\sigma}{\simeq}"]&
			\pi_{\widetilde{\beta}_1,\cdots,\widetilde{\beta}_n}\\
		      C^{\kk}(\sll_{n|1})
			\arrow[r, " \simeq "] \arrow[u,symbol=\subset, " \Psi\ "]&
			 C^{\check{\kk}}(\sll_{n},\OO_{\mathrm{sr}}) \arrow[u,symbol=\subset, " \check{\Psi}\ "].
\end{tikzcd}
\end{center}
Here, we have used isomorphisms 
\begin{align*}
    \Com(\pi_J,\pi^{\kk+\dC}_{\h}\otimes V_\Z)\simeq \pi_{\widetilde{\alpha}_1,\cdots,\widetilde{\alpha}_n},\quad \Com (\pi_{H_\Delta}, (\pi^{\check{\kk}+\check{\dC}}_{\check{\h}} \otimes \mathcal{A}^\times) \otimes V_\Z)\simeq \pi_{\widetilde{\beta}_1,\cdots,\widetilde{\beta}_n}
\end{align*}
by setting 
\begin{align*}
\widetilde{\alpha}_i=\begin{cases}{h_i}  & (i\neq n )\\ h_n-(\kk+\dC)x &(i=n),\end{cases} \quad 
\widetilde{\beta}_i=\begin{cases} \check{h}_i  & (i\neq n-1,n )\\ \check{h}_n-(\check{\kk}+\check{\dC})(x_1+y_1) &(i=n-1)\\ x_1 & (i=n).\end{cases} 
\end{align*}
By Lemma \ref{FS duality for FF}, the isomorphism $\sigma$ sends 
\begin{align}\label{eq: identification of Heisenberg}
    \widetilde{\alpha}_i \mapsto -\frac{1}{\check{\kk}+\check{\dC}}\widetilde{\beta}_i,\quad (i=1,\cdots, n).
\end{align}
Let $\pi_{\widetilde{\alpha}_1,\cdots,\widetilde{\alpha}_n}^{\dO}$ and $\pi_{\widetilde{\beta}_1,\cdots,\widetilde{\beta}_n}^\dso$ denote the Heisenberg vertex algebras over 
$$\dO\xrightarrow{\simeq} \dso,\quad \x\mapsto \frac{1}{\dx}$$
generated by $\widetilde{\alpha}_i$ ($i=1,\cdots,n$) and $\frac{1}{\dx}\widetilde{\beta}_i$ ($i=1,\cdots,n$), respectively Then the isomorphism $\sigma$ is lifted to 
\begin{align*}
    \sigma^{\dO}\colon  \pi_{\widetilde{\alpha}_1,\cdots,\widetilde{\alpha}_n}^{\dO} \xrightarrow{\simeq} \pi_{\widetilde{\beta}_1,\cdots,\widetilde{\beta}_n}^\dso
\end{align*}
whose evaluation $-\otimes_{\dO}\C_{\kk}$, equivalently $-\otimes_{\dso}\C_{\check{\kk}}$, recovers the isomorphism $\sigma$. 
Now, by applying the evaluation $-\otimes_{\dO}\C_{0}$, equivalently $-\otimes_{\dso}\C_{\infty}$, we obtain 
    \begin{center}
\begin{tikzcd}[ column sep = huge]
			\pi^0_{\h}
			\arrow[r, " \overset{\sigma}{\simeq}"]&
			\pi^\infty_{\beta_1,\cdots,\beta_n}\\
		      C^{\kk_c}(\sll_{n|1})
			\arrow[r, " \simeq "] \arrow[u,symbol=\subset, " \Psi\ "]&
			 C^{\infty}(\sll_{n},\OO_{\mathrm{sr}}) \arrow[u,symbol=\subset, " \check{\Psi}\ "].
\end{tikzcd}
\end{center}
Here, $\pi^\infty_{\beta_1,\cdots,\beta_n}$ is the differential algebra generated 
the evaluations $\beta_i$ of $\widetilde{\beta}_i$ at $\check{\kk}=\infty$, where $C^{\infty}(\sll_{n},\OO_{\mathrm{sr}})$ naturally embeds as $\pi^\infty_{\beta_1,\cdots,\beta_n}$ is a subalgebra of $(\sA^\dso \otimes_{\dso}\pi^{\dso}_{\check{\h}})\otimes_\dso \C_\infty$ in \eqref{eq: Wakimoto intergal form} through the localization
\begin{align*}
    \sA^\dso \hookrightarrow \sA^{\times \dso}
\end{align*}
where $\sA^{\times \dso}$ is generated by $(\frac{1}{\dx}a)^{\pm1}$ and $a^*$ over $\dso$.
This completes the proof.
\endproof

\section{Character formula and plane partitions}\label{sec: Character formula}
In this section, we derive the character formula of the center $\cent(V^{\kk_c}(\gl_{n|1}))$.
Let us introduce the Pochhammer symbol
\begin{align*}
    (z;q)_n=(1-z)(1-zq)\cdots(1-zq^{n-1})
\end{align*}
for $n=0,1,\cdots,\infty$ and the multi-variable one 
\begin{align*}
    (z_1,\cdots,z_m;q)_n=(z_1;q)_n\cdots (z_m;q)_n.
\end{align*}

\begin{theorem}\label{thm: character formula of the center}
For all $n\geq1$, the following formula holds:
\begin{align*}
    \ch[\cent(V^{\kk_c}(\gl_{n|1}))](q)=\frac{\sum_{m=0}^\infty (-1)^mq^{\frac{1}{2}m(m+1)}(q^{m+1};q)_{n-1}}{(q;q)_\infty^2(q,\cdots,q^{n-1};q)_\infty}.
\end{align*}
\end{theorem}
\begin{proof}
By Theorem \ref{them: Detecting the FF center} (1) and Theorem \ref{thm: Wakimoto for sCoset} (2), there is an isomorphism of quasi-conformal (i.e., graded) vertex algebras 
\begin{align*}
    \cent(V^{\kk_c}(\gl_{n|1}))\simeq C^{\kk_c}(\gl_{n|1}).
\end{align*}
By the decomposition \eqref{eq: Heisenberg decomposition}, $C^{\kk_c}(\gl_{n|1})\simeq C^{\kk}(\gl_{n|1})$ as graded vector spaces for generic levels $\kk$, and thus 
\begin{align*}
    \ch[\cent(V^{\kk_c}(\gl_{n|1}))](q)= \ch[C^{\kk_c}(\gl_{n|1})](q)
\end{align*}
For convenience, we compute the character on the right-hand side via the duality Theorem \ref{FS duality}. Since 
$$\ch[\W(\gl_{n},\OO_{\mathrm{sr}})](q)=\frac{1}{(q,q,\cdots,q^{n-1},zq,z^{-1}q^{n-1};q)_\infty},$$
we have 
\begin{align}
    \nonumber\ch[C^{\kk_c}(\gl_{n|1})](q)
    \nonumber&=\res{\frac{1}{(q,\cdots,q^{n-1},zq,z^{-1}q^{n-1};q)_\infty}}\\
    \label{eq: first manipulation}&=\frac{1}{(q,\cdots,q^{n-1};q)_\infty}\res{\frac{1}{(zq,z^{-1}q^{n-1};q)_\infty}}, 
\end{align}
where $\underset{z=0}{\mathrm{Res}}f(z)\frac{dz}{z}=f_0$ denotes the residue of the series $f(z)=\sum_n f_n z^n$. We note that the formula \eqref{eq: first manipulation} for $n=1$ holds by using the $\beta\gamma$-system instead of $\W(\gl_{n},\OO_{\mathrm{sr}})$, thanks to \cite{CL4}. 
It follows from the expansion
\begin{align*}
    \frac{1}{(zq,z^{-1}q;q)_\infty}=\frac{1}{(q;q)_\infty^2}\sum_{s\in \Z}\left(\Phi_s(q)-\Phi_{s-1}(q)\right)z^s
\end{align*}
by the unary false theta functions $\Phi_s(q)=\sum_{m=0}^\infty (-1)^mq^{\frac{1}{2}m(m+1)}q^{ms}$ \cite{Andrews84} and the expansion of the symmetric functions 
\begin{align}\label{sq: symmetric poly}
(1+x_1w)\cdots (1+x_Nw)=\sum_{a=0}^{N}e_a(x_1,\cdots,x_N)w^a,    
\end{align}
we obtain 
\begin{align*}
    &\res{\frac{1}{(zq,z^{-1}q^{n-1};q)_\infty}}
    =\res{\frac{(z^{-1}q;q)_{n-2}}{(zq,z^{-1}q;q)_\infty}}\\
    &=\frac{1}{(q;q)_\infty^2}\sum_{a=0}^{n-2}(\Phi_a(q)-\Phi_{a-1}(q))(-1)^a e_a(q,\cdots,q^{n-2})\\
    &=\frac{1}{(q;q)_\infty^2}\sum_{m=0}^\infty(-1)^mq^{\frac{1}{2}m(m+1)}(1-q^{-m}) \sum_{a=0}^{n-2}q^{ma}(-1)^a e_a(q,\cdots,q^{n-2})\\
    &\overset{(1)}{=}\frac{1}{(q;q)_\infty^2}\sum_{m=0}^\infty(-1)^mq^{\frac{1}{2}m(m+1)}(1-q^{-m}) 
    (1-q^{m+1})\cdots (1-q^{m+n-2})\\
    &\overset{(2)}{=}\frac{1}{(q;q)_\infty^2}\sum_{m=0}^\infty(-1)^mq^{\frac{1}{2}m(m+1)}(1-q^{m+1}) 
    (1-q^{m+2})\cdots (1-q^{m+n-1})\\
    &=\frac{1}{(q;q)_\infty^2}\sum_{m=0}^\infty(-1)^mq^{\frac{1}{2}m(m+1)}(q^{m+1};q)_{n-1}.
    \end{align*}
Here, we have used \eqref{sq: symmetric poly} in (1); we have changed the sum $\sum_{m=0}^{\infty}$ into $\sum_{m=1}^{\infty}$, as the term $m=0$ is zero, and then $\sum_{m'=0}^{\infty}$ by setting $m=m'+1$ in (2).
By plugging the right-hand side into \eqref{eq: first manipulation}, we obtain the desired formula.
\end{proof}

The $q$-series appearing on the right-hand side of Theorem \ref{thm: character formula of the center} is understood as the generating function of plane partitions with a pit condition found in the theory of $q$-deformations of (universal) $\W$-algebras, called the quantum toroidal algebra $\widehat{\widehat{\gl}}_1$ \cite{BFM18, FM12}. 
Let us briefly recall the plane partitions.
A plane partition a of a nonnegative integer $N$ (or a plane partition of
weight $N$) is, by definition, a set of non-negative integers $\lambda=(\lambda_{i,j})_{i,j>0}$, satisfying the conditions
\begin{itemize}
    \item[(1)] $\lim\lambda_{i,j}=0$ as $i\rightarrow \infty$ or $j\rightarrow \infty$;
    \item[(2)] $\lambda_{i,j}\geq \lambda_{i,j+1}$ and $\lambda_{i,j}\geq \lambda_{i+1,j}$ for all $i,j$.
\end{itemize}
The sum $|\lambda|=\sum_{i,j}\lambda_{i,j}$ is called the weight of the plane partition $\lambda$. The set of plane partitions is denoted by $\text{PP}$. 

Note that each plane partitions $\lambda=(\lambda_{i,j})_{i,j>0}$ of weight $N$ can be represented by a matrix of size at most $N$ whose $(i,j)$-entry is $\lambda_{i,j}$. The pit condition $(i,j)\notin \lambda$ is, by definition, the condition $a_{i,j}=0$, or, equivalently $\lambda_{i',j'}=0$ for all $i'\geq i$ and $j'\geq j$ by (2).
Pictorially, the plane partitions can be realized as the 3d Young diagrams, i.e., the boxes of size one on the domain $\{(x,y,z)\in \R^3\mid x,y,z\geq0\}$ by placing $a_{i,j}$ boxes on $[i-1,i]\times [j-1,j]$ on the $xy$-plane. 
In the below, we present a plane partition $\lambda$ and the corresponding 3d Young diagram. The plane partition $\lambda$ satisfies the pit conditions $(1,4), (3,2), (4,1)\notin \lambda$, which are drawn as the black shadows on the $xy$-plane for the 3d Young diagram.
\begin{figure}[h]
\centering
\begin{minipage}{0.45\textwidth}
\[\lambda=\left[\begin{array}{cccc} 2 &2& 1 &0 \\1 &1 &1 &0 \\ 1&0 &0 & 0 \\ 0&0&0&0\end{array} \right]
\]
\end{minipage}
\hspace{1cm}
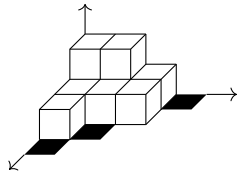
\begin{minipage}{0.45\textwidth}
\centering
\begin{tikzpicture}[scale=0.2]
    \fill[black] (0,0)--(2,0)--(3,1)--(1,1)--(0,0);
    \fill[black] (3,1)--(5,1)--(6,2)--(4,2)--(3,1);
    \fill[black] (9,3)--(11,3)--(12,4)--(10,4)--(9,3);
	\draw (1,1)--(1,3)--(3,5)--(3,7)--(4,8);
    \draw (3,1)--(3,3)--(5,5)--(5,7)--(6,8);
    \draw (3,1)--(4,2)--(4,4);
    \draw (6,2)--(6,4)--(7,5)--(7,7)--(8,8);
    \draw (7,5)--(8,6)--(8,8);
    \draw (8,2)--(8,4)--(10,6);
    \draw (8,2)--(10,4)--(10,6);
    \draw (9,3)--(9,5);
    %%%%%%%%%%%%%%%%%%%%%%%%%%%
    \draw (1,1)--(3,1);
    \draw (1,3)--(3,3);
    \draw (2,4)--(8,4);
    \draw (4,2)--(8,2);
    \draw (3,5)--(9,5);
    \draw (3,7)--(7,7);
    \draw (4,8)--(8,8);
    \draw (8,6)--(10,6);
    \draw[->] (0,0)--(-1,-1);
    \draw[->] (4,8)--(4,10);
	\draw[->] (12,4)--(14,4);
	\end{tikzpicture}
\end{minipage}
\end{figure}

The classical theorem by MacMahon (see e.g.\ \cite[Chapter 7]{St99}) asserts that the generating function of the plane partition is given by the formula:
\begin{align*}
    \sum_{\lambda\in \mathrm{PP}} q^{|\lambda|}=\frac{1}{\displaystyle{\prod_{n>0}(1-q^n)^n}}.
\end{align*}
The generating functions of the plane partitions subject to the pit conditions were obtained by one of the authors \cite{FM12} and are expressed by the following formulas.

\begin{theorem}[\cite{BFM18, FM12}]
The generating function $F_{m,n}(q)$ of the plane partitions with a pit condition $(m+1,n+1)$ with $n\geq m$ is given by 
\begin{align*}
F_{m,n}(q)=\frac{1}{(q;q)_\infty^{n+m}}&\sum_{k_1\geq\cdots \geq k_m\geq 0} \Big( (-1)^{\sum k_i}q^{\frac{1}{2}\sum k_i(k_i+2i-1)}\\
&\hspace{1cm} \times \prod_{1\leq i<j \leq n}(1-q^{k_i-k_j+j-i})\prod_{1\leq i<j \leq m}(1-q^{k_i-k_j+j-i})\Big),
\end{align*}
where we set $k_i=0$ for $i>m$. 
\end{theorem}
Let us consider the pit condition $(2,n+1)$ by setting $m=1$. Then, we have 
\begin{align*}
    F_{1,n}(q)
    &=\frac{1}{(q;q)_\infty^{n+1}}\sum_{k_1\geq 0}(-1)^{k_1}q^{\frac{1}{2}k_1(k_1+1)}\prod_{1\leq i<j \leq n}(1-q^{k_i-k_j+j-i}).
\end{align*}
Since 
\begin{align*}
\prod_{1\leq i<j \leq n}(1-q^{k_i-k_j+j-i})
&=\prod_{1\leq i< n}(1-q^{k_1+j-i}) \times \prod_{2\leq i<j \leq n}(1-q^{j-i})\\
&=(q^{k_1+1};q)_{n-1}\times (q;q)_{n-2}(q^2;q)_{n-3}\cdots (q^{n-2};q)_{1},
\end{align*}
we have
\begin{align*}
    F_{1,n}(q)
    \frac{\sum_{k_1=0}^\infty (-1)^{k_1}q^{\frac{1}{2}k_1(k_1+1)}(q^{k_1+1};q)_{n-1}}{(q;q)_\infty^2(q,\cdots,q^{n-1};q)_\infty},
\end{align*}
which agrees with the right-hand side of the equality in Theorem \ref{thm: character formula of the center}.
Therefore, we obtain the following corollary, generalizing the case $n=1$ proven by Molev and Mukhin in \cite{MM}.
\begin{corollary}\label{cor: characters as plane partition counting} For all $n\geq 1$, 
    \begin{align*}
        \ch[\cent(V^{\kk_c}(\gl_{n|1}))](q)=\sum_{\begin{subarray}c \lambda \in \mathrm{PP}\\ (2,n+1)\notin \lambda \end{subarray}} q^{|\lambda|}.
    \end{align*}
\end{corollary}

\section{Regular $\W$-superalgebras beyond the \tmath{$\gl_{n|1}$} case}\label{sec: Some conjectures}
\subsection{Conjectures}
To describe the center $\cent(V^{\kappa_c}(\gl_{n|1}))$, the regular $\W$-superalgebra $\W^\kappa(\gl_{n|1})$ and its duality to the subregular $\W$-algebra $\W^{\check{\kappa}}(\gl_{n},\OO_{\sub})$ played a crucial role. 
A more general duality has been proposed for other regular $\W$-superalgebras $\W^\kappa(\gl_{n|m})$ in \cite{CFLN}.
We may assume $n\geq m$ and let $\W^{\check{\kappa}}(\gl_{n},\OO_{[n-m,1^m]})$ denote the $\W$-algebra associated with the nilpotent orbit $\OO_{[n-m,1^m]}$ corresponding to the Jordan block $[n-m,1^m]$. 
\begin{conjecture}[\cite{CFLN}] \label{Ito conjecture}
For irrational levels $\kappa,\check{\kappa} \notin \Q$ satisfying the relation 
\begin{align*}
    (\kk+n-m)(\check{\kk}+n)=1,
\end{align*}
there is an isomorphism of vertex superalgebras 
\begin{align*}
    \W^{\kk}(\gl_{n|m})\simeq \Com \left(V^{\check{\kk}_{\circ}}(\gl_{m}), \W^{\check{\kk}}(\gl_{n},\OO_{[n-m,1^m]})\otimes \mathcal{F}_{n|m} \right)
\end{align*}
where the $\widehat{\gl}_{m}$-action is the diagonal and 
\begin{align*}
    \mathcal{F}_{n|m}=\begin{cases}
        V_{\Z^m}, & (n>m),\\
        \mathcal{A}^{m}\otimes V_{\Z^m},  & (n=m),
    \end{cases}
    \quad \check{\kk}_{\circ}=\check{\kk}+(n-m).
\end{align*}
\end{conjecture}
Note that the dual of the critical level $\kappa=\kappa_c$ is $\check{\kk}=\infty$, which is naturally interpreted as the limit of $\W^{\check{\kk}}(\gl_{n},\OO_{[n-m,1^m]})$ to the Poisson vertex supalgebra 
$$\W^{\infty}(\gl_{n},\OO_{[n-m,1^m]})=\C[J_\infty \mathscr{S}_{\OO}].$$
Here, the right-hand side is the algebra of regular functions over the jet (super)scheme associated to the Slodowy slice $\mathscr{S}_{\OO}\subset \gl_n$ for the orbit $\OO=\OO_{[n-m,1^m]}$. The slice $\mathscr{S}_{\OO}$ is a Poisson affine superspace equipped with $\mathrm{GL}_m$-action induced by a natural embedding $\gl_m\subset \mathscr{S}_{\OO}$. The $\mathrm{GL}_m$-action decomposes $\mathscr{S}_{\OO}$ into 
\begin{align*}
    \mathscr{S}_{\OO} \simeq \A^{n-m} \oplus \gl_m \oplus \C^m \oplus \overline{\C}^m
\end{align*}
where $\A^{n-m}$ is the $(n-m)$-dimensional trivial representation and $\C^m$ is the natural representation with $\overline{\C}^m$ its dual.
The subspace $\A^{n-m}$ generates a commutative vertex algebra $\C[J_\infty \A^{n-m}]$ while $\C^m \oplus \overline{\C}^m$ does a commutative one $\C[J_\infty (\C^m \oplus \overline{\C}^m)]$ equipped with a $\GL_m$-action, which we identify with the degenerate $\beta\gamma$-system vertex algebra $\sA^m_{\mathrm{deg}}$. 
The $\gl_m$-structure on $\C[J_\infty \mathscr{S}_{\OO}]$ motivates the following conjecture, the critical level analogue of Conjecture \ref{Ito conjecture}.

\begin{conjecture}\label{conj: reg Wsalg}
There is an isomorphism of vertex superalgebras 
\begin{align*}
    \W^{\kk_c}(\gl_{n|m}) \simeq \C[J_\infty \A^{n-m}] \otimes (\sA^m_{\mathrm{deg}} \otimes \mathcal{F}_{n|m})^{\GL_m}.
\end{align*}
\end{conjecture}
As a consequence, the centers of $\W^{\kk_c}(\gl_{n|m})$ and $V^{\kk_c}(\gl_{n|m})$ should have the following description, thanks to \cite{A}.

\begin{conjecture}\label{critical Ito conjecture}
There are isomorphisms of (commutative) vertex algebras 
\begin{align*}
     \cent(V^{\kk_c}(\gl_{n|m})) \simeq  \cent(\W^{\kk_c}(\gl_{n|m})) \simeq \C[J_\infty \A^{n-m}] \otimes \sA^m_{\mathrm{deg}}{}^{\GL_m}.
\end{align*}
\end{conjecture}

We note that this conjecture agrees with our previous conjecture \cite{AN-2024} on the description of $\cent(V^{\kk_c}(\gl_{n|m}))$.
In the remainder of this section, we prove these conjectures for small-rank cases.

\subsection{Case of \tmath{$\gl_{3|2}$}}
Based on \cite{KW04}, it is straightforward to compute the following presentation of $\W^{\kk_c}(\sll_{3|2})$ by the generators and relations.

\begin{proposition}
   The $\W$-superalgebra $\W^{\kk_c}(\sll_{3|2})$ admits freely generated by  
\begin{align*}
    J, S_2,L,S_3\ \mathrm{(even)},\qquad  s_1^\pm, s_2^\pm\ \mathrm{(odd)}
\end{align*}
with $S_2, S_3$ central, satisfying the following $\lambda$-brackets
\begin{align*}
    &[J {}_\lambda J]=2\lambda,\quad [J {}_\lambda L]=0,\quad [L{}_\lambda L]=\partial L+ 2L\lambda +\frac{1}{12}\lambda^3,\\
    &\\
    & [J{}_\lambda s^\pm_{i}]=\pm s^\pm_{i}\quad (i=1,2),\\
       & [L{}_\lambda s^\pm_{1}]=s^\pm_{2}+\tfrac{1}{4} s^\pm_{1}\lambda,\quad  [L{}_\lambda s^\pm_{2}]=Ls^\pm_{1}+\tfrac{5}{4} s^\pm_{1} \lambda+\tfrac{1}{4} s^\pm_{1}\lambda^2,\\
    &\\
    &[s^\pm_{1}{}_\lambda s^\pm_{1}]=0,\quad [s^\pm_{1}{}_\lambda s^\pm_{2}]=0,
        \quad [s^\pm_{2}{}_\lambda s^\pm_{2}]=\tfrac{1}{4}s_1^\pm \pd s_1^\pm-s_1^\pm{}s_2^\pm,\\
    &[s_1^+{}_\lambda s_1^-]= S_2\\
        &[s_1^+{}_\lambda s_2^-]=(\tfrac{2}{9}S_3+\tfrac{1}{6}\partial S_2+\tfrac{1}{2}JS_2-\tfrac{1}{2}s_1^+s_1^-)+\tfrac{1}{2}S_2\lambda,\\
        &[s_2^+{}_\lambda s_1^-]= -(\tfrac{2}{9}S_3+\tfrac{1}{6} \partial S_2+\tfrac{1}{2}JS_2-\tfrac{1}{2}s_1^+s_1^-)-\tfrac{1}{2}S_2\lambda,\\
        &[s_2^+{}_\lambda s_2^-]=(-\tfrac{1}{4}J \partial S_2-S_2L-\tfrac{1}{2}s_1^-s_2^+-\tfrac{1}{2}s_1^+s_2^-+\tfrac{1}{4}s_1^+\partial s_1^--\tfrac{1}{4}S_2 \partial J-\tfrac{1}{9}\partial S_3+\tfrac{1}{24}\partial^2S_2)\\
        &\hspace{2cm}+(-\tfrac{2}{9}S_3-\tfrac{1}{6} \partial S_2-\tfrac{1}{2}JS_2+\tfrac{1}{4}s_1^+s_1^-)\lambda-\tfrac{3}{8}S_2\lambda^2. 
\end{align*}
\end{proposition}
On the other hand, recall from \cite[Proposition 5.2]{A} that the vertex superalgebra $(\sA^2_{\mathrm{deg}} \otimes V_{\Z^2})^{\GL_2}$ is strongly generated by 
\begin{align*}
    &J_{0,p}=-\sum \psi_i\pd^p \psi^*_i,\quad J_{1,p}=\sum a_i\pd^p a_i^*, \\
    &J_{2,p}=-\sum \psi_i\pd^p a_i^*,\quad J_{3,p}=\sum a_i\pd^p \psi^*_i.
\end{align*}
for $ 0 \leq p <2$ where the sum is taken over $i=1,2$.
We introduce the (quasi-) conformal grading on $\sA^2_{\mathrm{deg}} \otimes V_{\Z^2}$ by the formula 
\begin{align*}
 &\Delta(a_i)=\Delta(a_i^*)=1,\quad \Delta(\psi_i)=\Delta(\psi_i^*)=1/2.
\end{align*}
Since the $\GL_2$-action preserves the grading, it induces a grading on $(\sA^2_{\mathrm{deg}} \otimes V_{\Z^2})^{\GL_2}$. The character is given by the following formula by using the Pochhammer symbol
$(a_1,\cdots,a_n;q)_\infty=\prod_{i=1}^n \prod_{p=0}^\infty(1-a_iq^p).$
\begin{lemma}\label{lem: character formula}
    $$\ch[(\sA^2_{\mathrm{deg}}\otimes V_{\Z^2})^{\GL_2}]=\frac{(-q^{3/2},-q^{3/2},-q^{5/2},-q^{5/2};q)_\infty}{(q,q^2,q^2,q^3;q)_\infty}.$$
\end{lemma}
\proof
The weights of the $\GL_2$-action on $\sA^2_{\mathrm{deg}}\otimes V_{\Z^2}$ is given by 
$$\wt(a_i )=\wt(\psi_i)=z_i,\quad \wt(a_i^*)=\wt(\psi_i^*)=-z_i$$
extended through the formulas 
$$\wt(\partial A)=\wt(A),\quad \wt(AB)=\wt(A)+\wt(B).$$
Then the character $\ch[(\sA^2_{\mathrm{deg}}\otimes V_{\Z^2})^{\GL_2}]$ is expressed as
\begin{align*}
    &\ch[(\sA^2_{\mathrm{deg}}\otimes V_{\Z^2})^{\GL_2}]\\
    &=\int \frac{\dz_1\dz_2}{z_1^2z_2^1}(z_1-z_2) \frac{(-z_1q^{1/2},-z_1^{-1}q^{1/2},-z_2q^{1/2},-z_2^{-1}q^{1/2};q)_\infty}{(z_1q,z_1^{-1}q,z_2q,z_2^{-1}q;q)_\infty},
\end{align*}
thanks to Schur's orthogonality theorem of characters.
Now, by using the trigonometric $\beta$-integral
\begin{align*}
    \frac{(zq^{1/2},z^{-1}q^{1/2};q)_\infty}{(azq^{1/2},bz^{-1}q^{1/2};q)_\infty}=\sum_{n\in \Z} (-1)^n\frac{(aq^{n+1},bq^{-n+1};q)_\infty}{(q,abq;q)_\infty}q^{n^2/2}z^{-n},
\end{align*}
the right-hand side is rewritten as follows:
\begin{align*}
    \ch[(\sA^2_{\mathrm{deg}}\otimes V_{\Z^2})^{\GL_2}]
    &=\frac{(-q^{3/2},-q^{3/2};q)_\infty^2}{(q,q^2;q)_\infty^2}-q \frac{(-q^{1/2},-q^{5/2};q)_\infty^2}{(q,q^2;q)_\infty^2}\\
    &=\frac{(-q^{3/2},-q^{3/2},-q^{5/2},q^{5/2};q)_\infty}{(q,q^2,q^2,q^3;q)_\infty}.
\end{align*}
\endproof
Now, we are ready to show Conjecture \ref{conj: reg Wsalg} for $\gl_{3|2}$.
\begin{theorem}
There is an isomorphism of vertex superalgebras
\begin{align*}
    \begin{array}{ccc}
    \Psi\colon \W^{\kk_c}(\sll_{3|2})  & \rightarrow &  (\sA^2_{\mathrm{deg}} \otimes V_{\Z^2})^{\GL_2}
    \end{array}
\end{align*}
such that 
\begin{align*}
    &J\mapsto -J_{0,0},\quad L \mapsto J_{0,1}-\tfrac{1}{4}J_{0,0}^2-\tfrac{1}{2} \partial J_{0,0}, \\
    &S_2 \mapsto -J_{1,0},\quad S_3 \mapsto -\tfrac{9}{2}J_{1,1}+\tfrac{3}{4} \partial J_{1,0} \\
    &s_1^+\mapsto J_{2,0}, \quad s_2^+\mapsto -J_{2,1}+\tfrac{1}{2}J_{0,0}J_{2,0}+\partial J_{2,0},\\
    &s_1^-\mapsto J_{3,0},\quad s_2^-\mapsto J_{3,1}-\tfrac{1}{2}J_{0,0}J_{3,0}.
\end{align*}
Thus, $ \W^{\kk_c}(\gl_{3|2})\simeq  \C[J_\infty \mathbb{A}]\otimes (\sA^2_{\mathrm{deg}} \otimes V_{\Z^2})^{\GL_2}$ holds as vertex superalgebras.
\end{theorem}
\proof
It is straightforward to show that $\Psi$ is a homomorphism of vertex superalgebras by checking the $\lambda$-brackets. 
As $J_{i,j}$ ($i=0,\cdots,3$, $j=0,1$) strongly generate $(\sA^2_{\mathrm{deg}} \otimes V_{\Z^2})^{\GL_2}$, we obtain that $\Psi$ is surjective by comparing the leading terms of the images.
To show that $\Psi$ is an isomoprhism, we introduce the (quasi-) conformal weights on $\W^k(\sll_{3|2})$ by the formula 
\begin{align*}
 &\Delta(J)=1,\quad \Delta(S_2)=\Delta(L)=2,\quad \Delta(S_3)=3,\\
 &\Delta(s_1^\pm)=3/2,\quad \Delta(s_2^\pm)=5/2.
\end{align*}
Since these strong generators induce the PBW basis, the character is given by
\begin{align*}
    \ch[\W^k(\sll_{3|2},\OO_{[3|2]})]=\frac{(-q^{3/2},-q^{3/2},-q^{5/2},-q^{5/2};q)_\infty}{(q,q^2,q^2,q^3;q)_\infty}.
\end{align*}
It is clear that $\Psi$ preserves the grading and thus the assertion follows since 
\begin{align*}
    \ch[\W^k(\sll_{3|2},\OO_{[3|2]})]=\ch[(\sA^2_{\mathrm{deg}} \otimes V_{\Z^2})^{\GL_2}]
\end{align*}
by Lemma \ref{lem: character formula} and $\Psi$ is surjective.
Now, the assertion for $\W^{\kk_c}(\gl_{3|2})$ is clear as 
$$\W^{\kk_c}(\gl_{3|2})\simeq \C[J_\infty \mathbb{A}]\otimes \W^{\kk_c}(\sll_{3|2}).$$
This completes the proof.
\endproof

\end{document}